\documentclass[11pt,oneside,onecolumn,reqno]{amsart}
\usepackage{amsgen, amstext,amsbsy,amsopn, amsthm, amsfonts,amssymb,amscd,amsmat
	h,euscript,enumerate,url,verbatim,calc,oldgerm,mathrsfs,bm,multicol,rotating,multirow,tabularx}

\usepackage{tikz}
\tikzset{->-/.style={decoration={
			markings,
			mark=at position #1 with {\arrow{>}}},postaction={decorate}}}
\usetikzlibrary{shapes,backgrounds,calc}
\usepackage[utf8]{inputenc}
\usepackage{latexsym}
\usepackage{graphics}
\usepackage{color}
\usepackage{bm}
\newcommand{\wrt}{with respect to}

\theoremstyle{plain}
\newtheorem{theorem}{Theorem}[section]

\theoremstyle{definition}

\newtheorem{remark}[theorem]{Remark}
\newtheorem{note}[theorem]{Note}

\theoremstyle{remark}
\newtheorem*{ack*}{\textbf{Acknowledgement}}
\usetikzlibrary{decorations.markings}

\usepackage{ragged2e}
\usepackage{pifont}
\usepackage{setspace}
\makeatletter

\newcommand{\Rmnum}[1]{\expandafter\@slowromancap\romannumeral #1@}
\makeatother
\usepackage{csquotes} 
\usepackage{cite}
\usepackage{amsaddr}
\usepackage{caption}
\usepackage{subcaption}
\usepackage{pgfplots}
\begin{document}
	\title[Delta shock waves for the Chaplygin Cargo-LeRoux model]{Riemann Problem for the isentropic Cargo-LeRoux system with extended Chaplygin fluids}
	\author[M. Zafar]{Priyanka$^\mathrm{A}$, \, M. Zafar$^\mathrm{A}$,\, V. D. Sharma$^\mathrm{B}$}
	\address{$^\mathrm{A}$ Department of Mathematics \& Computing, Dr. B.R. Ambedkar NIT Jalandhar, India\\ $^\mathrm{B}$ Department of Mathematics, IIT Bombay, Mumbai, India}
%	\email[Priyanka]{priyanka.ma.21@nitj.ac.in}
	\email[M.Zafar]{zafarm@nitj.ac.in}
%	\email[V. D. Sharma]{vduttsharma@gmail.com}
	\begin{abstract}
%		 We investigate Riemann solutions of the isentropic Cargo–LeRoux system of conservation laws incorporating extended Chaplygin fluid models. The system admits rich nonlinear structures, including delta shock waves and vacuum states. We rigorously analyze the formation and asymptotic behavior of these solutions, particularly in the vanishing-parameters limit of the equation of state. Our results reveal new insights into singular wave dynamics and highlight the role of generalized pressure laws in the evolution of weak solutions.
		This study deals with the Riemann solutions of the isentropic Cargo–LeRoux system of conservation laws along with an extended Chaplygin fluid model and a generalized equation of state that captures nonlinear compressibility and non-ideal fluid behavior. It is shown that in the vanishing pressure limiting case, the Riemann solution yields a delta shock and a vacuum state. Delta shocks form when mass concentrates due to vanishing pressure while vacuum states emerge in rarefaction dominated regions. These wave patterns are crucial in modeling real world problems like discontinuous and multi-phase astrophysical fluid flows, underwater explosions, and combustion fronts. 
		 
	\end{abstract}
	\maketitle
		\section{Introduction}
		The Riemann problems for Euler equations play a key role in modeling discontinuities in fluid variables; their solutions involve rarefaction waves, contact discontinuities, and shock waves (see, \cite{whitham, smoller2012shock}). Rarefaction waves occur when pressure and density drop across a discontinuity; these are relevant in vacuum formation. Contact discontinuities occur when density changes but velocity and pressure remain constant across the interface; these model interfaces between different materials or phases. Shock waves represent compression phenomena like sonic booms, blast waves, and piston driven flows, where discontinuous jumps satisfy Rankine-Hugoniot jump conditions.

		In the recent past, several studies have investigated Riemann solutions in a variety of material mediums under different pressure laws (see, \cite{chen2003,cheng2018,cheng2019,sheng2020,ibrahim2022,Kuila2022,sahoo2023} and many more references cited therein). 		
%		Chen and Liu (see, \cite{chen2003}) showed the appearance of delta shock and vacuum state in the vanishing pressure limit of polytropic Euler equations, while Jiang et al. (see, \cite{jiang2023cavitation}) examined limiting behaviors in isothermal dusty gas equations. Li and Shen (see, \cite{li2025a}), and Sun and Xin (see, \cite{sun2024}), analyzed Riemann solutions of the Aw-Rascle-Zhang traffic model under different pressure laws, highlighting concentration effects. Guo et al. (see, \cite{guo2017}) and Zhu et al. (see, \cite{zhu2024}) explored the vanishing pressure limit in Chaplygin and modified Chaplygin Euler equations with a source term, respectively. Yang and Liu (see, \cite{yang2015}) used flux approximations to reveal concentration and cavitation in nonisentropic Euler systems. Additional studies by Sheng and Shao (see, \cite{sheng2020}) and Sun and Wei (see, \cite{sun2023}) addressed vanishing pressure and two-phase flow models, respectively. Further references on Riemann problem can be found in \cite{tong2025, neves2025,  shen2022z,  ibrahim2022, zhang2022delta, li2024, sarrico2024}.
%	
		In this paper, we analyze the structure of Riemann solutions to the following initial value problem for isentropic Cargo-LeRoux model
			\begin{equation}\label{m1}
			\begin{cases}
				\rho_{t}+(\rho u)_{x}=0,
				\\(\rho q)_{t}+(\rho q u )_{x}=0,
				\\(\rho u)_{t}+(\rho u^{2} + \mathcal{P}(\rho) +q )_{x}=0,
			\end{cases} 
		\end{equation} 
	with initial data of the form
	\begin{equation} \label{m3}
		(\rho, q,  u)(x, t=0) =	\begin{cases}
			(\rho_{\ell}, q_{\ell}, u_{\ell}), \quad x<0,
			\\ (\rho_{r}, q_{r}, u_{r}), \quad x>0,
		\end{cases}
	\end{equation}
		together with the generalized equation of state for extended Chaplygin fluids
	\begin{equation}\label{m2}
		\mathcal{P}(\rho) = A \rho^{n}-\frac{B}{\rho^\alpha}, \quad 1\leq n \leq 3, \qquad 0<\alpha\leq1,
	\end{equation}where $\rho$ is the density, $u$ the velocity, $\mathcal{P}$ the pressure, $q(x, t)$ hydro-static pressure as a potential satisfying $ q_{x}=  \rho g$ and $q_{t}=- \rho g u,$ with gravitational constant $g$, and $A,$ $ B$ are non-negative constants. It may be noticed that \eqref{m2} for $n = 1$ corresponds to the modified Chaplygin fluid and for $A = 0,$ it reduces to the generalized Chaplygin fluid (see, \cite{benaoum2002,setare2007}).
	The Cargo-LeRoux model \eqref{m1} was introduced by Cargo and LeRoux (see, \cite{cargo1994, leroux1999}) to develop a well balanced numerical algorithm for Euler equations with constant gravity. A distinctive feature of this model is the incorporation of an additional pressure-type variable ($q$), through which the effects of gravity are incorporated into a system of conservation laws. This formulation not only preserves the essential hyperbolic character of the underlying fluid equations but also connects the Cargo–LeRoux system with multi-pressure Euler models and related nonlinear fluid systems.
	% Such a structure makes the model relevant to gravity-driven flows and wave propagation arising in geophysical, hydraulic, and environmental settings.
	The Cargo-LeRoux model
	describes the behavior of a fluid flowing under the influence of constant gravity. Owing to its ability to capture the effects of gravity on fluid flow, the model has applications in a wide range of disciplines, viz., geophysics, aeronautical engineering, civil engineering, biomedical engineering, and environmental engineering etc. 
	In particular, it provides a mathematical framework for studying fluid transport in systems such as pipelines, channels, and open-water flows. Moreover, the model is useful for analyzing the propagation of waves in fluids, which is relevant to the design of coastal and hydraulic structures, flood-risk assessment, and the study of large-scale fluid phenomena such as ocean currents and atmospheric flows (see, \cite{gupta2026pressureless, maurya2024conservation}). 
	It may be remarked that if $q$ is regarded as a constant then the system  \eqref{m1} reduces to the Euler system without external forces; focusing on the formation and propagation of non-classical wave phenomena such as delta shock waves and vacuum states.

	The above equation of state \eqref{m2} is a generalization that captures nonlinear compressibility and non-ideal fluid behavior; indeed, it is a powerful tool for modeling non-ideal compressible fluids, especially in regimes where standard polytropic laws fail. It captures two competing effects; for instance in astrophysical fluid flows, the term $A\rho^n$ models stiffening at high densities while the term $-B\rho^ {(-\alpha)}$ models exotic behavior like vacuum formation or clustering; this term becomes dominant in the limit $\rho\to 0$ leading to negative pressure. This dual behavior is key to modeling nonlinear phenomena like astrophysical flows, combustion and detonation, underwater explosions, and implosions. Euler equations with this equation of state give rise to rich wave structures including delta shocks, which represent singular solutions where mass and momentum concentrate into a singular Dirac delta; these arise in zero pressure limits, modeling sticky particles; mathematically it means that a Dirac delta function appears in the solution.

		Cheng constructed the Riemann solutions of pressureless Cargo-LeRoux model and isentropic Chaplygin Cargo-LeRoux model with three and four types  of structures, respectively (see, \cite{cheng2018, cheng2019}). Kuila and his collaborators studied the Riemann solutions and their interactions for Cargo-LeRoux model with different pressure terms (see, \cite{jana2022, Kuila2022, sahoo2023, jana2025}). Kumozec and Nedeljkov (see, \cite{kumozec2024}) discussed the Riemann solutions along with shadow waves for isentropic Cargo-LeRoux model with generalized Chaplygin gas.
	Cheng and Yang (see, \cite{cheng2011}) showed the emergence of Dirac delta functions in nonlinear chromatography equations, while Jiang et al. (see, \cite{jiang2023cavitation}) examined limiting behaviors in isothermal dusty gas equations. Li and Shen (see, \cite{li2025a}), and Sun and Xin (see, \cite{sun2024}), analyzed Riemann solutions of the Aw-Rascle-Zhang traffic model under different pressure laws, highlighting concentration effects. Tong and Shen (see, \cite{tong2019}) explored the vanishing pressure limit in extended Chaplygin Euler equations, and Yang and Liu (see, \cite{yang2015}) used flux approximations to reveal concentration and cavitation in nonisentropic Euler systems. Additional studies by Sheng and Shao (see, \cite{sheng2020}) and Sun and Wei (see, \cite{sun2023}) addressed vanishing pressure and two-phase flow models, respectively. Further references on delta shock waves can be found in \cite{tong2025, neves2025,  shen2022z,  sheng2015, yang2016flux,guo2017,zhang2022delta,li2024,sarrico2024}.

	  The existence of delta shock solution is crucial in modeling real world problems like discontinuous and multi-phase astrophysical fluid flows. In the limit $A,B \to 0,$ the equation of state \eqref{m2} simplifies toward pressureless gas dynamics leading to delta shocks and vacuum states; delta shocks form when mass concentrates due to vanishing pressure while vacuum states emerge in rarefaction dominated regions. So, the generalized equation of state \eqref{m2} is crucial for modeling non-ideal compressible fluids, capturing nonlinear compressibility and negative pressure effects. These concepts come in handy for studying nonlinear phenomena stated as above. Riemann solutions, discussed in the following sections, reveal rich wave dynamics; delta shocks are particularly interesting because they involve singular solutions, often occurring when the pressure term vanishes or becomes negligible. These have practical applications in zero-pressure gas dynamics, underwater explosions, and combustion fronts.

%	   Cheng (see, \cite{cheng2018, cheng2019}) in his motivational work constructed Riemann solutions for pressureless and isentropic Chaplygin Cargo-LeRoux models, identifying multiple wave structures. Kuila et al. (see, \cite{jana2022,Kuila2022,sahoo2023}) studied wave interactions under varying pressure conditions, while Kumozec and Nedeljkov (see, \cite{kumozec2024}) analyzed Riemann solutions and shadow waves in generalized Chaplygin settings. 

	This paper is organized as follows: In section $ 2,$ we have constructed the Riemann solutions of IVP \eqref{m1}-\eqref{m2}. The limiting behaviors of the Riemann solutions are studied in section $3.$
%	The numerical simulations are presented in Section 5.
	Section $4$ presents the numerical simulations that validate the formation of single delta shock solution in the limiting case.

\section{\mbox{Riemann Solutions for the Extended Chaplygin Cargo-LeRoux model}}  
\begin{comment}
		The system \eqref{m1}, in matrix form, can be written as
	\begin{equation}\label{m7}
		\begin{bmatrix}
			\rho 
			\\ q \\u
		\end{bmatrix}_{t} + \begin{bmatrix}
			u &0 & \rho
			\\	0 & u & 0 \\ \frac{\mathcal{P}^{\prime}(\rho)}{\rho} & \frac{1}{\rho} & u
		\end{bmatrix}
		\begin{bmatrix}
			\rho 
			\\ q \\ u
		\end{bmatrix}_{x} = 0,\end{equation}
	where
	\begin{equation*}
		\mathcal{P}^{\prime}(\rho) = nA\rho^{n-1} + \frac{\alpha B}{\rho^{\alpha+1}}.
	\end{equation*} \end{comment}

	The system \eqref{m1} together with the equation of state \eqref{m2} is a strictly hyperbolic system as the Jacobian matrix has the following  three distinct eigenvalues\vspace{-0.2cm}
	\begin{equation*}
		\lambda_{1}=u-\sqrt{\mathcal{P}^{\prime}(\rho)},\quad  \lambda_{2}=u,
		  \quad\lambda_{3}=u+\sqrt{\mathcal{P}^{\prime}(\rho)} , 
	\end{equation*}
	with the corresponding eigenvectors
	\begin{equation*}
		\bm{\overrightarrow{r_{1}}}=(-\rho, 0, \sqrt{\mathcal{P}^{\prime}(\rho)})^{T}, \quad \bm{\overrightarrow{r_{2}}}=(-1, \mathcal{P}^{\prime}(\rho), 0)^{T}, \quad \bm{\overrightarrow{r_{3}}}=(\rho, 0, \sqrt{\mathcal{P}^{\prime}(\rho)})^{T}, 
	\end{equation*} where $	\mathcal{P}^{\prime}(\rho) = nA\rho^{n-1} + \frac{\alpha B}{\rho^{\alpha+1}}.$ Also, 
	\begin{equation*}
		  \bigtriangledown\lambda_{1}\cdot\bm{\overrightarrow{r_{1}}} \neq 0,  \,  \bigtriangledown\lambda_{2}\cdot\bm{\overrightarrow{r_{2}}}=0, \, \text{and} \, \bigtriangledown\lambda_{3}\cdot\bm{\overrightarrow{r_{3}}}\neq0.
	\end{equation*}
It may be noticed  that the first and third characteristic fields are genuinely nonlinear whereas the second characteristic field is linearly degenerate. Hence, the waves associated with $ \lambda_{1} $ and $ \lambda_{3} $ are either shocks or rarefaction waves and the wave associated with $ \lambda_{2} $ is a contact discontinuity.
	\subsection{Rarefaction wave and contact discontinuity}
Performing the self-similar transformation $\zeta=x/t,$ the Riemann problem presented by \eqref{m1}-\eqref{m2} transforms into the following boundary value problem
	\begin{equation}\label{m8}
		\begin{cases}
			-\zeta\rho_{\zeta} + (\rho u)_{\zeta} = 0,
			\\	-\zeta(\rho q)_{\zeta} + (\rho q u)_{\zeta} = 0,\\	-\zeta(\rho u)_{\zeta} + (\rho u^{2}+\mathcal{P}(\rho)+q)_{\zeta} = 0, 
		\end{cases}
	\end{equation}
	and
	\begin{equation}\label{m9}	(\rho, q, u) = 
		\begin{cases}
			(\rho_{\ell}, q_{\ell},  u_{\ell}),\quad\zeta=-\infty,
			\\  (\rho_{r}, q_{r}, u_{r}),\quad\zeta=+\infty.
		\end{cases}
	\end{equation}
	
%	For any smooth solution, the system \eqref{m8} in simplified form can be written as
%	\begin{equation}\label{m10}
%		\begin{bmatrix}
%			u -\zeta & 0 & \rho 
%			\\ 0 &u-\zeta & 0 \\ \frac{\mathcal{P}^{\prime}(\rho)}{\rho} &\frac{1}{\rho} & u -\zeta
%		\end{bmatrix} \begin{bmatrix}
%			\rho_{\zeta}\\ q_{\zeta} \\ u_{\zeta}
%		\end{bmatrix} = 0.
%	\end{equation}
%\begin{figure}
%	\centering
%	\small
%	\begin{tikzpicture}
%		\tikzset{->-/.style={decoration={
%					markings,
%					mark=at position .5 with {\arrow{>}}},postaction={decorate}}}
%		\draw [thick,<-](-3.7,5.1)node [right]{$t$} --(-3.7,3.9);
%		\draw[thick, ->] (-2,0)--(6,0)node [below]{$ x $};
%		\draw [thick, blue] (2,0) to (5.1,2.5);
%		\draw [thick, green] (2,0) to (-1.1,2.5);
%		\draw [thick, dashed, orange] (2,0) to (2.9,3.67);
%		\node[left] at (1.2,0.55) {$(\rho_{\ell}, q_{\ell}, u_{\ell})$};
%		\node[left, rotate=325] at (2,1) {$(\rho_{*_{1}}, q_{*_{1}}, u_{*_{1}})$};
%		\node[left, rotate=55] at (3.8,2.7) {$(\rho_{*_{2}}, q_{*_{2}}, u_{*_{2}})$};
%		\node[left] at (4.8,0.55) {$(\rho_{r}, q_{r}, u_{r})$};
%		\node[below] at (2,0) {O};
%		\node[left, rotate=48, align=center] at (0.2,4.7) {Shock wave($ \mathcal{S}_{1} $)\\or Rarefaction wave($ \mathcal{R}_{1} $)};
%		\node[left, rotate=348, align=center] at (4.5,3.8) {Contact \\discontinuity($ \mathcal{J}_{2} $)};
%		\node[left, rotate=320, align=center] at (8,0.8) {Shock wave($ \mathcal{S}_{3} $)\\or Rarefaction wave($ \mathcal{R}_{3} $)};
%	\end{tikzpicture}
%	\caption{Structure of the Riemann solution for \eqref{m1}-\eqref{m2}}
%	\label{m100}
%\end{figure}
\noindent In view of the system \eqref{m8}, we have
\\(i) for a given left state $ (\rho_{\ell}, q_{\ell}, u_{\ell}), $
the possible  states $(\rho, q, u)$ that can be connected to $ (\rho_{\ell}, q_{\ell}, u_{\ell})$
on the right by back rarefaction wave are as follows
\begin{equation}\label{m19}
	\mathcal{R}_{1}(\rho_{\ell}, q_{\ell}, u_{\ell}):\begin{cases}
		\zeta=\lambda_{1}=u-\sqrt{\mathcal{P}^{\prime}(\rho)},\\
		u=u_{\ell}-\int_{\rho_{\ell}}^{\rho} \frac{\sqrt{\mathcal{P}^{\prime}(s)}}{s} ds,
		\\ q=q_{\ell}, \quad \rho<\rho_{\ell},\quad u>u_{\ell},
	\end{cases}
\end{equation}(ii) for a fixed state $(\rho_{o}, q_{o}, u_{o}),$ any state that can be connected to $(\rho_{o}, q_{o}, u_{o})$ by contact discontinuity must be located on the following curve
\vspace{-0.2cm}
\begin{equation}\label{m23}
\mathcal{J}_{2}(\rho_{o}, q_{o}, u_{o})=\begin{cases}
	\zeta=\lambda_{2}=\sigma_{2}=u=u_{o}\\q+\mathcal{P}(\rho)=q_{o}+\mathcal{P}(\rho_{o}).
\end{cases}
\end{equation}
(iii) for a fixed right state $ (\rho_{r}, q_{r}, u_{r}), $
the possible states $(\rho, q, u)$ that can be connected to $ (\rho_{r}, q_{r}, u_{r})$
on the left by front rarefaction wave are as follows\vspace{-0.2cm}
\begin{equation}\label{m21}
	\mathcal{R}_{3}(\rho_{r}, q_{r}, u_{r}) : \begin{cases}
		\zeta=\lambda_{3}=u+\sqrt{\mathcal{P}^{\prime}(\rho)},\\
		u=u_{r}-\int_{\rho}^{\rho_{r}} \frac{\sqrt{\mathcal{P}^{\prime}(s)}}{s} ds,
		\\ q=q_{r}, \quad \rho_{r}>\rho,\quad u_{r}>u.
	\end{cases}
\end{equation}
\begin{comment}
	
	\noindent In view of the system \eqref{m8}, we have  $ 1 $-rarefaction wave \begin{equation}\label{m11}
	\mathcal{R}_{1}:\begin{cases}
		\zeta=\lambda_{1}=u-\sqrt{\mathcal{P}^{\prime}(\rho)},\\
		u=u_{\ell}-\int_{\rho_{\ell}}^{\rho} \frac{\sqrt{\mathcal{P}^{\prime}(s)}}{s} ds,
		\\ q=q_{\ell}, \quad \rho<\rho_{\ell},\quad u>u_{\ell},
	\end{cases}
\end{equation}
 or the contact discontinuity
	\begin{equation}\label{m12}
		\mathcal{J}_{2} : \begin{cases} \zeta=\lambda_{2}=u=u_{\ell},
		%	\\ u=u_{\ell},
			\\ q+\mathcal{P}(\rho)=q_{\ell}+\mathcal{P}(\rho_{\ell})
		\end{cases}
	\end{equation}
	or the $ 3 $-rarefaction wave 
	\begin{equation}\label{m13}
		\mathcal{R}_{3} : \begin{cases}
			\zeta=\lambda_{3}=u+\sqrt{\mathcal{P}^{\prime}(\rho)},\\
			u=u_{\ell}+\int_{\rho_{\ell}}^{\rho} \frac{\sqrt{\mathcal{P}^{\prime}(s)}}{s} ds,
			\\ q=q_{\ell}, \quad \rho>\rho_{\ell},\quad u>u_{\ell},
		\end{cases}
	\end{equation}
along with $\frac{d\lambda_{1}}{d \rho}<0$ and $\frac{d\lambda_{3}}{d \rho}>0.$ \end{comment}
%In addition, for  $ 1 $-rarefaction wave,  
%\begin{equation}
%	\frac{du}{d\rho}=-\frac{\sqrt{\mathcal{P}^{\prime}(\rho)}}{\rho}<0, \quad \text{and} \quad  \frac{d^{2}u}{d\rho^{2}}=\frac{-\rho \mathcal{P}^{\prime \prime}(\rho)+2 \mathcal{P}^{\prime}(\rho)}{2\rho^{2}\sqrt{\mathcal{P}^{\prime}(\rho)}}>0.
%\end{equation} Thus, the $ 1 $-rarefaction wave is strictly monotonically decreasing and convex in nature. Similarly, the $ 3 $-rarefaction wave is strictly monotonically increasing and concave in nature.

	\subsection{Shock wave}
	For a bounded discontinuity $\zeta= \sigma,$ the following Rankine-Hugoniot jump relations hold\vspace{-0.2cm}
	\begin{equation}\label{m14}
		\begin{cases}
			-\sigma[\rho]+[\rho u]=0, \\ -\sigma[\rho q]+[\rho q u]=0,\\ 	-\sigma[\rho u]+[\rho u^2 + \mathcal{P}(\rho)+q]=0,
		\end{cases}
	\end{equation}
	where $ [f]=f_{r}-f_{\ell}$ etc.
In view of \eqref{m14} and the lax entropy conditions for back and front shock waves, we have 
\\(i) for a given left state $ (\rho_{\ell}, q_{\ell}, u_{\ell}), $
the possible  states $(\rho, q, u)$ that can be connected to $ (\rho_{\ell}, q_{\ell}, u_{\ell})$
on the right by back shock wave are as follows
\begin{equation}\label{m20}
	\mathcal{S}_{1}(\rho_{\ell}, q_{\ell}, u_{\ell})=\begin{cases}
		\sigma_{1}=u-\sqrt{\frac{\rho_{\ell}(\mathcal{P}(\rho)-\mathcal{P}(\rho_{\ell}))}{\rho(\rho-\rho_{\ell})}}=u_{\ell}-\sqrt{\frac{\rho(\mathcal{P}(\rho)-\mathcal{P}(\rho_{\ell}))}{\rho_{\ell}(\rho-\rho_{\ell})}},\\u=u_{\ell}+(\rho_{\ell}-\rho)\sqrt{\frac{(\mathcal{P}(\rho)-\mathcal{P}(\rho_{\ell}))}{\rho_{\ell}\rho(\rho-\rho_{\ell})}} \\q_{\ell}=q,\quad \rho_{\ell}<\rho,
	\end{cases}
\end{equation}
(ii) for a fixed right state $ (\rho_{r}, q_{r}, u_{r}), $
the possible states $(\rho, q, u)$ that can be connected to $ (\rho_{r}, q_{r}, u_{r})$
on the left by front shock wave are as follows
\begin{equation}\label{m22}
	\mathcal{S}_{3}(\rho_{r}, q_{r}, u_{r})=\begin{cases}
		\sigma_{3}=u_{r}+\sqrt{\frac{\rho(\mathcal{P}(\rho_{r})-\mathcal{P}(\rho))}{\rho_{r}(\rho_{r}-\rho)}}=u+\sqrt{\frac{\rho_{r}(\mathcal{P}(\rho_{r})-\mathcal{P}(\rho))}{\rho(\rho_{r}-\rho)}},\\u_{r}=u-(\rho-\rho_{r})\sqrt{\frac{(\mathcal{P}(\rho_{r})-\mathcal{P}(\rho))}{\rho\rho_{r}(\rho_{r}-\rho)}} \\q=q_{r},\quad \rho>\rho_{r}.
	\end{cases}
\end{equation}  
%Also, $ 1 $-shock wave (or $ 2 $-shock wave) is monotonically decreasing (increasing) and convex (concave) in nature.

\subsection{Construction of the Riemann solution}

\noindent Let's take the projections of $ \mathcal{R}_{1}, $ $ \mathcal{S}_{1}, $ $ \mathcal{R}_{3}  $ and $ \mathcal{S}_{3}  $ on the $ (\rho, u) $-plane and denote them by $\mathcal{R}_{q}^{\ell}, $ $ \mathcal{S}_{q}^{\ell}, $ $ \mathcal{R}_{q}^{r}  $ and $ \mathcal{S}_{q}^{r},  $ respectively. In addition, denote $\mathscr{W}_{q}^{\ell}=\mathcal{R}_{q}^{\ell}\cup \mathcal{S}_{q}^{\ell}$ and $\mathscr{W}_{q}^{r}=\mathcal{R}_{q}^{r}\cup \mathcal{S}_{q}^{r}.$ Then each point $ (\rho, u)\in \mathscr{W}_{q}^{\ell}$ satisfies
\begin{equation}\label{m24}
	u=\mathtt{g}_{1}(\rho)=\begin{cases}
		u_{\ell}-\int_{\rho_{\ell}}^{\rho} \frac{\sqrt{\mathcal{P}^{\prime}(s)}}{s} ds,
		 \qquad \quad\qquad \rho \leq\rho_{\ell},\\
		 u_{\ell}+(\rho_{\ell}-\rho)\sqrt{\frac{(\mathcal{P}(\rho)-\mathcal{P}(\rho_{\ell}))}{\rho_{\ell}\rho(\rho-\rho_{\ell})}},\quad \rho_{\ell}<\rho,
	\end{cases}
\end{equation}and every point $ (\rho, u)\in \mathscr{W}_{q}^{r}$ satisfies
\begin{equation}\label{m25}
u=\mathtt{g}_{2}(\rho)=\begin{cases}
	u_{r}-\int_{\rho}^{\rho_{r}} \frac{\sqrt{\mathcal{P}^{\prime}(s)}}{s} ds,
	\qquad \quad\qquad \rho \leq\rho_{r},\\
	u_{r}+(\rho-\rho_{r})\sqrt{\frac{(\mathcal{P}(\rho_{r})-\mathcal{P}(\rho))}{\rho\rho_{r}(\rho_{r}-\rho)}},\quad \rho_{r}<\rho.
\end{cases}
\end{equation}
%One can easily see that the elementary wave curve $ \mathtt{g}_{1}(\rho) $ is strictly monotonically decreasing and convex in nature whereas the elementary wave curve $ \mathtt{g}_{2}(\rho) $ is strictly monotonically increasing and concave in nature.
One can easily see that the elementary wave curve $ \mathtt{g}_{1}(\rho) $ is strictly monotonically decreasing in nature whereas the elementary wave curve $ \mathtt{g}_{2}(\rho) $ is strictly monotonically increasing in nature. The monotonicity of the functions $ \mathtt{g}_{1}(\rho) $  and $ \mathtt{g}_{2}(\rho) $ ensure the existence of the inverse functions $ \rho = \mathtt{g}_{1}^{-1}(u) $ and $ \rho = \mathtt{g}_{2}^{-1}(u), $ respectively. In $ (\rho, u)$-plane, let $ (\rho_{k}, u_{k}) $ denote a point of intersection between $ \mathscr{W}_{q}^{\ell}$ and $ \mathscr{W}_{q}^{r}$. Further, we define
\begin{equation}\label{m26}
	\mathtt{h}_{1}(\rho)=\mathcal{P}(\rho)+q_{\ell}, \qquad \mathtt{h}_{2}(\rho)=\mathcal{P}(\rho)+q_{r},
\end{equation}and \begin{equation}\label{m40}
\mathtt{f}_{1}(u)=\mathtt{h}_{1}(\mathtt{g}_{1}^{-1}(u))=\mathcal{P}(\mathtt{g}_{1}^{-1}(u))+q_{\ell},\qquad \mathtt{f}_{2}(u)=\mathtt{h}_{2}(\mathtt{g}_{2}^{-1}(u))=\mathcal{P}(\mathtt{g}_{2}^{-1}(u))+q_{r}.
\end{equation} Clearly, $ \mathtt{f}_{1}(u) $ and $ \mathtt{f}_{2}(u) $
are monotonically decreasing and increasing functions, respectively. In addition, one can have\vspace{-0.2cm}
\begin{equation}\label{m27}
	 \begin{cases}
	 	\mathtt{f}_{1}(-\infty)= \mathtt{h}_{1}(\mathtt{g}_{1}^{-1}(-\infty))=+\infty,\\
	 	\mathtt{f}_{2}(-\infty)= \mathtt{h}_{2}(\mathtt{g}_{2}^{-1}(-\infty))=-\infty,\\
	 	\mathtt{f}_{1}(+\infty)= \mathtt{h}_{1}(\mathtt{g}_{1}^{-1}(+\infty))=-\infty,\\
	 	\mathtt{f}_{2}(+\infty)= \mathtt{h}_{2}(\mathtt{g}_{2}^{-1}(+\infty))=+\infty.
	 \end{cases}
	 \end{equation}
	 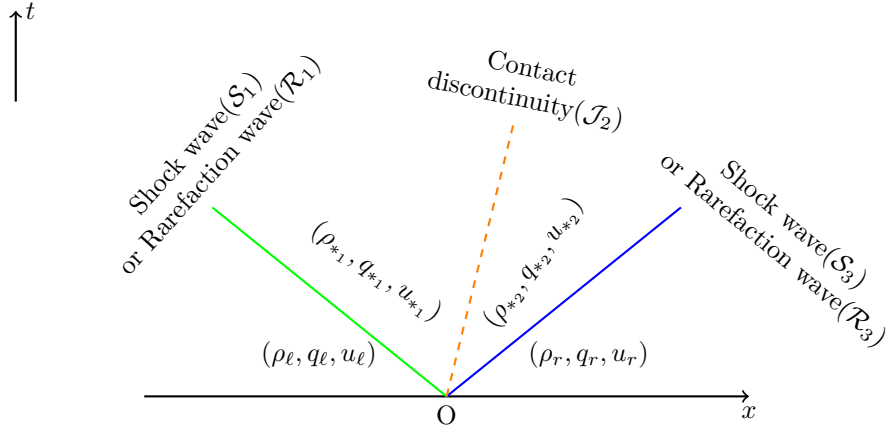
\begin{figure}[h!]
	 	\centering
	 	\small
	 	\begin{tikzpicture}
	 		\tikzset{->-/.style={decoration={
	 					markings,
	 					mark=at position .5 with {\arrow{>}}},postaction={decorate}}}
	 		\draw [thick,<-](-3.7,5.1)node [right]{$t$} --(-3.7,3.9);
	 		\draw[thick, ->] (-2,0)--(6,0)node [below]{$ x $};
	 		\draw [thick, blue] (2,0) to (5.1,2.5);
	 		\draw [thick, green] (2,0) to (-1.1,2.5);
	 		\draw [thick, dashed, orange] (2,0) to (2.9,3.67);
	 		\node[left] at (1.2,0.55) {$(\rho_{\ell}, q_{\ell}, u_{\ell})$};
	 		\node[left, rotate=325] at (2,1) {$(\rho_{*_{1}}, q_{*_{1}}, u_{*_{1}})$};
	 		\node[left, rotate=55] at (3.8,2.7) {$(\rho_{*_{2}}, q_{*_{2}}, u_{*_{2}})$};
	 		\node[left] at (4.8,0.55) {$(\rho_{r}, q_{r}, u_{r})$};
	 		\node[below] at (2,0) {O};
	 		\node[left, rotate=48, align=center] at (0.2,4.7) {Shock wave($ \mathcal{S}_{1} $)\\or Rarefaction wave($ \mathcal{R}_{1} $)};
	 		\node[left, rotate=348, align=center] at (4.5,3.8) {Contact \\discontinuity($ \mathcal{J}_{2} $)};
	 		\node[left, rotate=320, align=center] at (8,0.8) {Shock wave($ \mathcal{S}_{3} $)\\or Rarefaction wave($ \mathcal{R}_{3} $)};
	 	\end{tikzpicture}
	 	\caption{Structure of the Riemann solution for \eqref{m1}-\eqref{m2}}
	 	\label{m100}
	 \end{figure}
	 \begin{figure}
	 	\centering
	 	
	 	\begin{minipage}[b]{0.5\textwidth}
	 		\resizebox{0.8\textwidth}{!}{
	 			\begin{tikzpicture}
	 				\tikzset{->-/.style={decoration={
	 							markings,
	 							mark=at position .5 with {\arrow{>}}},postaction={decorate}}}
	 				
	 				\draw [ultra thick,<-](-4.5,5.3)node [left]{$u$} --(-4.5,-4.2);
	 				\draw[ultra thick, ->] (-4.5,-1.8)--(6,-1.8)node [below]{$ \rho $};
	 				\draw [ultra thick, magenta] (-4.34,-4.1) to [out=80,in=595](-3,-0.50);
	 				\draw [ultra thick, green] (-3,-0.50) to [out=53,in=550](5.4,4);
	 				\draw [ultra thick, orange] (-4.34,5) to [out=280,in=120](-3.2,1.5);
	 				\draw [ultra thick, blue] (-3.2,1.5) to [out=300,in=168](5.4,-3.9);
	 				\draw [ultra thick, dashed, gray] (-2.7,4.5) to [out=280,in=160](2.5,-1);
	 				\draw [ultra thick, dashed, gray] (-2,-1) to [out=60,in=560](3.5,3);
	 				\node[left] at (-3.6,4.9) {$ \mathcal{R}_{q}^{\ell}$};
	 				\node[right] at (4.85,4.25) {$ \mathcal{S}_{q}^{r}$};
	 				\node[above] at (-3.95,-4.4) {$ \mathcal{R}_{q}^{r}$};
	 				\node[above] at (5.1,-4.6) {$ \mathcal{S}_{q}^{\ell}$};
	 				\node at (-3.2,1.5) {$ \bullet $};
	 				\node at (-2.6,4.7) {$ \mathtt{f}_{1}(u) $};
	 				\node at (3.4,2.6) {$ \mathtt{f}_{2}(u) $};
	 				\draw [ultra thick, dashed, red] (-4.5,0.75) to (3,0.75);
	 				\node at (-2.38,0.27) {$ \bullet $};
	 				\node at (-2.74,0.73) {$ \bullet $};
	 				\node at (-1.9,0.73) {$ \bullet $};	
	 				\node[left] at (-3,-0.55) {$ (\rho_{r}, u_{r}) $};
	 				\node at (-3.8,1) {$ (\rho_{*_{1}}, u_{*_{1}}) $};
	 				\node[above] at (-1.7,0.75) {$ (\rho_{*_{2}}, u_{*_{2}}) $};
	 				\node at (-2.5,1.5) {$ (\rho_{\ell}, u_{\ell}) $};
	 				\node at (0.1,0.48) {$ (\mathtt{f}_{1}(\tilde{u}),  \tilde{u}) $};
	 				\node[left] at (-2.4,0.24) {$ (\rho_{k}, u_{k}) $};
	 				\node at (-3,-0.50) {$ \bullet $};
	 				\node at (-0.65,0.73) {$ \bullet $};
	 				
	 			\end{tikzpicture}
	 		}
	 	\end{minipage}\hfill
	 	\begin{minipage}[b]{0.4\textwidth}
	 		\resizebox{0.92\textwidth}{!}{
	 			\begin{tikzpicture}
	 				\tikzset{->-/.style={decoration={
	 							markings,
	 							mark=at position .5 with {\arrow{>}}},postaction={decorate}}}
	 				%	\draw [help lines] (-1,0) grid (10,10);
	 				\draw [ultra thick,<-](0,8.3)node [right]{$u$} --(0,6);
	 				\draw[ultra thick, ->] (0.6,0.8)--(8.3,0.8)node [below]{$ \rho$};
	 				\draw [ultra thick, blue] (2.4,6.1) to [out=285,in=145](3,5.2);
	 				\draw[ultra thick, dashed, red] (3,5.2) to (5.7,5.2);
	 				\draw[ultra thick, green](5.7,5.2)[out=220,in=430] to (4.2,3);
	 				\node at (3,5.2) {$ \bullet $};
	 				\node at (2.4,6.1) {$ \bullet $};
	 				\node at (3,5.2) {$ \bullet $};
	 				\node at (5.7,5.2) {$ \bullet $};
	 				\node at (4.2,3) {$ \bullet $};
	 				\node[left] at (3.1,5) {$ (\rho_{*_{1}}, u_{*_{1}}) $};
	 				\node[right] at (5.7,5.2) {$ (\rho_{*_{2}}, u_{*_{2}}) $};
	 				\node[right] at (4.2,3) {$ (\rho_{r}, u_{r}) $};
	 				\node[left] at (2.4,6.1) {$ (\rho_{\ell}, u_{\ell}) $};
	 				\node at (2.15,5.57) {$ \mathcal{S}_{q}^{\ell} $};
	 				\node[above] at (5.25,3.8) {$ \mathcal{S}_{q}^{r} $};
	 				\node[above] at (4.47,5.2) {$ \mathcal{J}_{2} $};
	 				\draw[ultra thick, ->] (4.4,5.2)--(4.47,5.2);
	 				\draw[ultra thick, ->] (4.81,4.25)--(4.77,4.2);
	 				\draw[ultra thick, ->] (2.59,5.6)--(2.62,5.55);
	 		\end{tikzpicture}}
	 		
	 	\end{minipage}
	 	
	 	\caption{Riemann solution whenever $ (\rho_{k}, u_{k})\in \mathcal{S}_{q}^{\ell}$, $ (\rho_{k}, u_{k})\in \mathcal{S}_{q}^{r}$, $ q_{\ell}>q_{r} $ and $\mathtt{f}_{1}(u_{\ell})<\mathtt{f}_{2}(u_{\ell})$}
	 	\label{m125}
	 \end{figure}

	 \begin{figure}
	 	\centering
	 	
	 	\begin{minipage}[b]{0.5\textwidth}
	 		\resizebox{0.8\textwidth}{!}{
	 			\begin{tikzpicture}
	 				\tikzset{->-/.style={decoration={
	 							markings,
	 							mark=at position .5 with {\arrow{>}}},postaction={decorate}}}
	 				
	 				\draw [ultra thick,<-](-4.5,5.3)node [left]{$u$} --(-4.5,-4.2);
	 				\draw[ultra thick, ->] (-4.5,-1.8)--(6,-1.8)node [below]{$ \rho $};
	 				\draw [ultra thick, magenta] (-4.34,-4.1) to [out=80,in=595](-3,-0.50);
	 				\draw [ultra thick, green] (-3,-0.50) to [out=53,in=550](5.4,4);
	 				\draw [ultra thick, orange] (-4.34,5) to [out=280,in=120](-3.2,1.5);
	 				\draw [ultra thick, blue] (-3.2,1.5) to [out=300,in=168](5.4,-3.9);
	 				\draw [ultra thick, dashed, gray] (-2.7,4.5) to [out=280,in=160](2.5,-1);
	 				\draw [ultra thick, dashed, gray] (-4,-1) to [out=60,in=560](2.5,4.4);
	 				\node[left] at (-3.6,4.9) {$ \mathcal{R}_{q}^{\ell}$};
	 				\node[right] at (4.85,4.25) {$ \mathcal{S}_{q}^{r}$};
	 				\node[above] at (-3.95,-4.4) {$ \mathcal{R}_{q}^{r}$};
	 				\node[above] at (5.1,-4.6) {$ \mathcal{S}_{q}^{\ell}$};
	 				\node at (-3.2,1.5) {$ \bullet $};
	 				\node at (-2.6,4.7) {$ \mathtt{f}_{1}(u) $};
	 				\node at (3,4.4) {$ \mathtt{f}_{2}(u) $};
	 				\draw [ultra thick, dashed, red] (-4.5,1.97) to (3,1.97);
	 				\node at (-3,-0.50) {$ \bullet $};
	 				\node at (-2.38,0.25) {$ \bullet $};
	 				\node at (-1.72,1.97) {$ \bullet $};	
	 				\node[right] at (-3.18,-0.75) {$ (\rho_{r}, u_{r}) $};
	 				\node at (-2.8,2.3) {$ (\rho_{*_{1}}, u_{*_{1}}) $};
	 				\node[above] at (0.2,1.33) {$ (\rho_{*_{2}}, u_{*_{2}}) $};
	 				\node at (-3.85,1.5) {$ (\rho_{\ell}, u_{\ell}) $};
	 				\node at (-1.25,2.25) {$ (\mathtt{f}_{1}(\tilde{u}),  \tilde{u}) $};
	 				\node[right] at (-2.35,0.24) {$ (\rho_{k}, u_{k}) $};
	 				\node at (-3.45,1.97) {$ \bullet $};
	 				\node at (-0.43,1.97) {$ \bullet $};		
	 		\end{tikzpicture}}
	 	\end{minipage}\hfill
	 	\begin{minipage}[b]{0.4\textwidth}
	 		\resizebox{0.92\textwidth}{!}{
	 			\begin{tikzpicture}
	 				\tikzset{->-/.style={decoration={
	 							markings,
	 							mark=at position .5 with {\arrow{>}}},postaction={decorate}}}
	 				%	\draw [help lines] (-1,0) grid (10,10);
	 				\draw [ultra thick,<-](0,8.4)node [right]{$u$} --(0,6);
	 				\draw[ultra thick, ->] (0.6,0.8)--(8.3,0.8)node [below]{$ \rho$};
	 				\draw [ultra thick, orange] (3,5.2) to [out=285,in=145](3.6,4.3);
	 				\draw[ultra thick, dashed, red] (3,5.2) to (6,5.2);
	 				\draw[ultra thick, green](6,5.2)[out=220,in=430] to (4.5,3);
	 				\node at (3,5.2) {$ \bullet $};
	 				\node at (3.6,4.3) {$ \bullet $};
	 				\node at (3,5.2) {$ \bullet $};
	 				\node at (6,5.2) {$ \bullet $};
	 				\node at (4.5,3) {$ \bullet $};
	 				\node[left] at (3,5.2) {$ (\rho_{*_{1}}, u_{*_{1}}) $};
	 				\node[right] at (6,5.2) {$ (\rho_{*_{2}}, u_{*_{2}}) $};
	 				\node[right] at (4.5,3) {$ (\rho_{r}, u_{r}) $};
	 				\node[left] at (3.6,4.2) {$ (\rho_{\ell}, u_{\ell}) $};
	 				\node at (2.8,4.7) {$ \mathcal{R}_{q}^{\ell} $};
	 				\node[above] at (5.45,3.8) {$ \mathcal{S}_{q}^{r} $};
	 				\node[above] at (4.47,5.2) {$ \mathcal{J}_{2} $};
	 				\draw[ultra thick, ->] (4.4,5.2)--(4.47,5.2);
	 				\draw[ultra thick, ->] (5.1,4.25)--(5.06,4.2);
	 				\draw[ultra thick, <-] (3.155,4.8)--(3.185,4.75);
	 		\end{tikzpicture}}	
	 	\end{minipage}
	 	
	 	\caption{Riemann solution whenever $ (\rho_{k}, u_{k})\in \mathcal{S}_{q}^{\ell}$, $ (\rho_{k}, u_{k})\in \mathcal{S}_{q}^{r}$, $ q_{\ell}>q_{r} $ and $\mathtt{f}_{1}(u_{\ell})>\mathtt{f}_{2}(u_{\ell})$}
	 	\label{m126}
	 \end{figure}

	 \begin{figure}
	 	\centering
	 	
	 	\begin{minipage}[b]{0.5\textwidth}
	 		\resizebox{0.8\textwidth}{!}{
	 			\begin{tikzpicture}
	 				\tikzset{->-/.style={decoration={
	 							markings,
	 							mark=at position .5 with {\arrow{>}}},postaction={decorate}}}
	 				
	 				\draw [ultra thick,<-](-4.5,5.3)node [left]{$u$} --(-4.5,-4.2);
	 				\draw[ultra thick, ->] (-4.5,-1.8)--(6,-1.8)node [below]{$ \rho $};
	 				\draw [ultra thick, magenta] (-4.34,-4.1) to [out=80,in=595](-3,-0.50);
	 				\draw [ultra thick, green] (-3,-0.50) to [out=53,in=550](5.4,4);
	 				\draw [ultra thick, orange] (-4.34,5) to [out=280,in=120](-3.2,1.5);
	 				\draw [ultra thick, blue] (-3.2,1.5) to [out=300,in=168](5.4,-3.9);
	 				\draw [ultra thick, dashed, gray] (-2.7,2.5) to [out=280,in=160](3,-2.8);
	 				\draw [ultra thick, dashed, gray] (-3,-2) to [out=60,in=560](2.5,2);
	 				\node[left] at (-3.6,4.9) {$ \mathcal{R}_{q}^{\ell}$};
	 				\node[right] at (4.85,4.25) {$ \mathcal{S}_{q}^{r}$};
	 				\node[above] at (-3.95,-4.4) {$ \mathcal{R}_{q}^{r}$};
	 				\node[above] at (5.1,-4.6) {$ \mathcal{S}_{q}^{\ell}$};
	 				\node at (-3.2,1.5) {$ \bullet $};
	 				\node at (-2.6,3) {$ \mathtt{f}_{1}(u) $};
	 				\node at (3.05,2.05) {$ \mathtt{f}_{2}(u) $};
	 				\draw [ultra thick, dashed, red] (-4.5,-0.13) to (3,-0.13);
	 				\node at (-2.38,0.27) {$ \bullet $};
	 				\node at (-2.7,-0.13) {$ \bullet $};
	 				\node at (-2.1,-0.13) {$ \bullet $};	
	 				\node[left] at (-3,-0.55) {$ (\rho_{r}, u_{r}) $};
	 				\node at (-1.8,-0.38) {$ (\rho_{*_{1}}, u_{*_{1}}) $};
	 				\node[above] at (-2,0.3) {$ (\rho_{k}, u_{k}) $};
	 				\node at (-3.85,1.5) {$ (\rho_{\ell}, u_{\ell}) $};
	 				\node at (-0.9,0.15) {$ (\mathtt{f}_{1}(\hat{u}),  \hat{u}) $};
	 				\node[left] at (-2.56,0.13) {$ (\rho_{*_{2}}, u_{*_{2}}) $};
	 				\node at (-3,-0.50) {$ \bullet $};
	 				\node at (-1.5,-0.13) {$ \bullet $};
	 				
	 			\end{tikzpicture}
	 		}
	 	\end{minipage}\hfill
	 	\begin{minipage}[b]{0.4\textwidth}
	 		\resizebox{0.92\textwidth}{!}{
	 			\begin{tikzpicture}
	 				\tikzset{->-/.style={decoration={
	 							markings,
	 							mark=at position .5 with {\arrow{>}}},postaction={decorate}}}
	 				%	\draw [help lines] (-1,0) grid (10,10);
	 				\draw [ultra thick,<-](0,8.5)node [right]{$u$} --(0,5.8);
	 				\draw[ultra thick, ->] (0.6,0.8)--(8.3,0.8)node [below]{$ \rho$};
	 				\draw [ultra thick, blue] (4.4,7) to [out=285,in=145](5.7,5.2);
	 				\draw[ultra thick, dashed, red] (3,5.2) to (5.7,5.2);
	 				\draw[ultra thick, green](3,5.2)[out=220,in=430] to (2.2,3.8);
	 				\node at (3,5.2) {$ \bullet $};
	 				\node at (4.4,7) {$ \bullet $};
	 				\node at (2.2,3.8) {$ \bullet $};
	 				\node at (5.7,5.2) {$ \bullet $};
	 				%\node at (4.2,3) {$ \bullet $};
	 				\node[left] at (2.95,5.25) {$ (\rho_{*_{2}}, u_{*_{2}}) $};
	 				\node[right] at (5.7,5.2) {$ (\rho_{*_{1}}, u_{*_{1}}) $};
	 				\node[left] at (2.2,3.8) {$ (\rho_{r}, u_{r}) $};
	 				\node[left] at (4.45,7) {$ (\rho_{\ell}, u_{\ell}) $};
	 				\node at (2.15,4.5) {$ \mathcal{S}_{q}^{r} $};
	 				\node[above] at (5.1,5.8) {$ \mathcal{S}_{q}^{\ell} $};
	 				\node[above] at (4.4,5.2) {$ \mathcal{J}_{2} $};
	 				\draw[ultra thick, <-] (4.4,5.2)--(4.47,5.2);
	 				\draw[ultra thick, ->] (4.78,6.1)--(4.82,6.05);
	 				\draw[ultra thick, <-] (2.49,4.6)--(2.53,4.65);
	 		\end{tikzpicture}}
	 		
	 	\end{minipage}
	 	
	 	\caption{Riemann solution whenever $ (\rho_{k}, u_{k})\in \mathcal{S}_{q}^{\ell}$, $ (\rho_{k}, u_{k})\in \mathcal{S}_{q}^{r}$, $ q_{\ell}<q_{r} $ and $\mathtt{f}_{1}(u_{r})>\mathtt{f}_{2}(u_{r})$}
	 	\label{m127}
	 \end{figure}

	 \begin{figure}
	 	\centering
	 	
	 	\begin{minipage}[b]{0.5\textwidth}
	 		\resizebox{0.8\textwidth}{!}{
	 			\begin{tikzpicture}
	 				\tikzset{->-/.style={decoration={
	 							markings,
	 							mark=at position .5 with {\arrow{>}}},postaction={decorate}}}
	 				
	 				\draw [ultra thick,<-](-4.5,5.3)node [left]{$u$} --(-4.5,-4.2);
	 				\draw[ultra thick, ->] (-4.5,-1.8)--(6,-1.8)node [below]{$ \rho $};
	 				\draw [ultra thick, magenta] (-4.34,-4.1) to [out=80,in=595](-3,-0.50);
	 				\draw [ultra thick, green] (-3,-0.50) to [out=53,in=550](5.4,4);
	 				\draw [ultra thick, orange] (-4.34,5) to [out=280,in=120](-3.2,1.5);
	 				\draw [ultra thick, blue] (-3.2,1.5) to [out=300,in=168](5.4,-3.9);
	 				\draw [ultra thick, dashed, gray] (-2.7,2.5) to [out=280,in=160](3,-2.8);
	 				\draw [ultra thick, dashed, gray] (-1.9,-2.7) to [out=60,in=560](3.2,1.5);
	 				\node[left] at (-3.6,4.9) {$ \mathcal{R}_{q}^{\ell}$};
	 				\node[right] at (4.85,4.25) {$ \mathcal{S}_{q}^{r}$};
	 				\node[above] at (-3.95,-4.4) {$ \mathcal{R}_{q}^{r}$};
	 				\node[above] at (5.1,-4.6) {$ \mathcal{S}_{q}^{\ell}$};
	 				\node at (-3.2,1.5) {$ \bullet $};
	 				\node at (-2.6,2.8) {$ \mathtt{f}_{1}(u) $};
	 				\node at (3.05,1.7) {$ \mathtt{f}_{2}(u) $};
	 				\draw [ultra thick, dashed, red] (-4.5,-0.95) to (3,-0.95);
	 				\node at (-2.38,0.27) {$ \bullet $};
	 				\node at (-3.27,-0.95) {$ \bullet $};
	 				\node at (-1.3,-0.95) {$ \bullet $};	
	 				\node[left] at (-2.87,-0.35) {$ (\rho_{r}, u_{r}) $};
	 				\node at (-1.8,-1.2) {$ (\rho_{*_{1}}, u_{*_{1}}) $};
	 				\node[above] at (-3.2,0) {$ (\rho_{k}, u_{k}) $};
	 				\node at (-3.85,1.5) {$ (\rho_{\ell}, u_{\ell}) $};
	 				\node at (-0.1,-0.7) {$ (\mathtt{f}_{1}(\hat{u}),  \hat{u}) $};
	 				\node[left] at (-2.67,-1.2) {$ (\rho_{*_{2}}, u_{*_{2}}) $};
	 				\node at (-3,-0.50) {$ \bullet $};
	 				\node at (-0.67,-0.95) {$ \bullet $};
	 				
	 			\end{tikzpicture}
	 		}
	 	\end{minipage}\hfill
	 	\begin{minipage}[b]{0.4\textwidth}
	 		\resizebox{0.92\textwidth}{!}{
	 			\begin{tikzpicture}
	 				\tikzset{->-/.style={decoration={
	 							markings,
	 							mark=at position .5 with {\arrow{>}}},postaction={decorate}}}
	 				%	\draw [help lines] (-1,0) grid (10,10);
	 				\draw [ultra thick,<-](0,10)node [right]{$u$} --(0,7.5);
	 				\draw[ultra thick, ->] (0.6,2.3)--(8.3,2.3)node [below]{$ \rho$};
	 				\draw [ultra thick, blue] (4.4,7) to [out=285,in=145](5.7,5.2);
	 				\draw[ultra thick, dashed, red] (3,5.2) to (5.7,5.2);
	 				\draw[ultra thick, magenta](3.5,6)[out=220,in=430] to (3,5.2);
	 				\node at (3,5.2) {$ \bullet $};
	 				\node at (4.4,7) {$ \bullet $};
	 				\node at (3.5,6) {$ \bullet $};
	 				\node at (5.7,5.2) {$ \bullet $};
	 				%\node at (4.2,3) {$ \bullet $};
	 				\node[left] at (3.1,5) {$ (\rho_{*_{2}}, u_{*_{2}}) $};
	 				\node[right] at (5.7,5.2) {$ (\rho_{*_{1}}, u_{*_{1}}) $};
	 				\node[left] at (3.6,6.2) {$ (\rho_{r}, u_{r}) $};
	 				\node[right] at (4.3,7.2) {$ (\rho_{\ell}, u_{\ell}) $};
	 				\node at (2.8,5.6) {$ \mathcal{R}_{q}^{r} $};
	 				\node[above] at (5.1,5.8) {$ \mathcal{S}_{q}^{\ell} $};
	 				\node[above] at (4.4,5.2) {$ \mathcal{J}_{2} $};
	 				\draw[ultra thick, <-] (4.4,5.2)--(4.47,5.2);
	 				\draw[ultra thick, ->] (4.78,6.1)--(4.82,6.05);
	 				\draw[ultra thick, ->] (3.19,5.6)--(3.22,5.65);
	 		\end{tikzpicture}}
	 		
	 	\end{minipage}
	 	
	 	\caption{Riemann solution whenever $ (\rho_{k}, u_{k})\in \mathcal{S}_{q}^{\ell}$, $ (\rho_{k}, u_{k})\in \mathcal{S}_{q}^{r}$, $ q_{\ell}<q_{r} $ and $\mathtt{f}_{1}(u_{r})<\mathtt{f}_{2}(u_{r})$}
	 	\label{m128}
	 \end{figure}
 \begin{theorem}
For any given initial data \eqref{m3}, there exists a unique pair of points $ (\rho_{*_{1}}, q_{*_{1}}, u_{*_{1}}) \in \mathscr{W}_{q}^{\ell} $ and $ (\rho_{*_{2}}, q_{*_{2}}, u_{*_{2}}) \in \mathscr{W}_{q}^{r} $ satisfying\vspace{-0.2cm}
\begin{equation*}
	u_{*_{1}}=u_{*_{2}}, \qquad h_{1}(\rho_{*_{1}})=h_{2}(\rho_{*_{2}}).
\end{equation*}
\end{theorem}
\begin{proof}
	The proof is divided into the following two cases:
	\\ $(i)$ If $q_{\ell}>q_{r},$ then $ \mathtt{f}_{1}(u_{k})=\mathtt{h}_{1}(\rho_{k})>\mathtt{h}_{2}(\rho_{k})= \mathtt{f}_{2}(u_{k}) $ implies that there exists a unique $ \tilde{u} \in (u_{k}, +\infty) $ such that $ \mathtt{f}_{1}(\tilde{u})= \mathtt{f}_{2}(\tilde{u}) .$ Thus, we have $ (\rho_{*_{1}}, u_{*_{1}})=(\mathtt{g}_{1}^{-1}(\tilde{u}), \tilde{u}) $ and $ (\rho_{*_{2}}, u_{*_{2}})=(\mathtt{g}_{2}^{-1}(\tilde{u}), \tilde{u}). $ 
		\\ $(ii)$ If $q_{\ell}<q_{r},$ then $ \mathtt{f}_{1}(u_{k})=\mathtt{h}_{1}(\rho_{k})<\mathtt{h}_{2}(\rho_{k})= \mathtt{f}_{2}(u_{k}) $ gives that there exists a unique $ \hat{u} \in (-\infty, u_{k}) $ such that $ \mathtt{f}_{1}(\hat{u})= \mathtt{f}_{2}(\hat{u}) .$ Hence, we have $ (\rho_{*_{1}}, u_{*_{1}})=(\mathtt{g}_{1}^{-1}(\hat{u}), \hat{u}) $ and $ (\rho_{*_{2}}, u_{*_{2}})=(\mathtt{g}_{2}^{-1}(\hat{u}), \hat{u}). $ 
\end{proof}	
 \begin{table}[h!]
	\scriptsize
	\centering
	\begin{tabularx}{\textwidth}{|l|X|X|X|X|}
		\hline
		& \bm{$\mathcal{A}:$} $(\rho_k, u_k) \in \mathcal{S}_q^\ell,\, \mathcal{S}_q^r$ & \bm{$\mathcal{B}:$} $(\rho_k, u_k) \in \mathcal{R}_q^\ell,\, \mathcal{R}_q^r$ & \bm{$\mathcal{C}:$} $(\rho_k, u_k) \in \mathcal{S}_q^\ell,\, \mathcal{R}_q^r$ & \bm{$\mathcal{D}:$} $(\rho_k, u_k) \in \mathcal{R}_q^\ell,\, \mathcal{S}_q^r$ \\
		\hline
		
		$q_\ell > q_r$ 
		&
		(i) $\mathtt{f}_1(u_\ell) < \mathtt{f}_2(u_\ell):\, \mathcal{S}_1 \mathcal{J}_2 \mathcal{S}_3$ \newline
		(ii) $\mathtt{f}_1(u_\ell) > \mathtt{f}_2(u_\ell):\, \mathcal{R}_1 \mathcal{J}_2 \mathcal{S}_3$
		&
		(i) $\mathtt{f}_1(u_r) < \mathtt{f}_2(u_r):\, \mathcal{R}_1 \mathcal{J}_2 \mathcal{R}_3$ \newline
		(ii) $\mathtt{f}_1(u_r) > \mathtt{f}_2(u_r):\, \mathcal{R}_1 \mathcal{J}_2 \mathcal{S}_3$
		&
		(i) $u_\ell \geq u_r,\ \mathtt{f}_1(u_r) < \mathtt{f}_2(u_r):\, \mathcal{S}_1 \mathcal{J}_2 \mathcal{R}_3$ \newline
		(ii) $u_\ell > u_r,\ \mathtt{f}_1(u_r) > \mathtt{f}_2(u_r),\ \mathtt{f}_1(u_\ell) < \mathtt{f}_2(u_\ell):\, \mathcal{S}_1 \mathcal{J}_2 \mathcal{S}_3$ \newline
		(iii) $u_\ell > u_r,\ \mathtt{f}_1(u_\ell) > \mathtt{f}_2(u_\ell):\, \mathcal{R}_1 \mathcal{J}_2 \mathcal{S}_3$ \newline
		(iv) $u_\ell < u_r,\ \mathtt{f}_1(u_\ell) < \mathtt{f}_2(u_\ell):\, \mathcal{S}_1 \mathcal{J}_2 \mathcal{R}_3$ \newline
		(v) $u_\ell < u_r,\ \mathtt{f}_1(u_\ell) > \mathtt{f}_2(u_\ell),\ \mathtt{f}_1(u_r) < \mathtt{f}_2(u_r):\, \mathcal{R}_1 \mathcal{J}_2 \mathcal{R}_3$ \newline
		(vi) $u_\ell \leq u_r,\ \mathtt{f}_1(u_r) > \mathtt{f}_2(u_r):\, \mathcal{R}_1 \mathcal{J}_2 \mathcal{S}_3$
		&
		$\mathcal{R}_1 \mathcal{J}_2 \mathcal{S}_3$
		\\
		\hline
		
		$q_\ell < q_r$
		&
		(i) $\mathtt{f}_1(u_r) > \mathtt{f}_2(u_r):\, \mathcal{S}_1 \mathcal{J}_2 \mathcal{S}_3$ \newline
		(ii) $\mathtt{f}_1(u_r) < \mathtt{f}_2(u_r):\, \mathcal{S}_1 \mathcal{J}_2 \mathcal{R}_3$
		&
		(i) $\mathtt{f}_1(u_\ell) > \mathtt{f}_2(u_\ell):\, \mathcal{R}_1 \mathcal{J}_2 \mathcal{R}_3$ \newline
		(ii) $\mathtt{f}_1(u_\ell) < \mathtt{f}_2(u_\ell):\, \mathcal{S}_1 \mathcal{J}_2 \mathcal{R}_3$
		&
		$\mathcal{S}_1 \mathcal{J}_2 \mathcal{R}_3$
		&
		(i) $u_\ell \geq u_r,\ \mathtt{f}_1(u_\ell) > \mathtt{f}_2(u_\ell):\, \mathcal{R}_1 \mathcal{J}_2 \mathcal{S}_3$ \newline
		(ii) $u_\ell > u_r,\ \mathtt{f}_1(u_r) > \mathtt{f}_2(u_r),\ \mathtt{f}_1(u_\ell) < \mathtt{f}_2(u_\ell):\, \mathcal{S}_1 \mathcal{J}_2 \mathcal{S}_3$ \newline
		(iii) $u_\ell > u_r,\ \mathtt{f}_1(u_r) < \mathtt{f}_2(u_r):\, \mathcal{S}_1 \mathcal{J}_2 \mathcal{R}_3$ \newline
		(iv) $u_\ell < u_r,\ \mathtt{f}_1(u_r) > \mathtt{f}_2(u_r):\, \mathcal{R}_1 \mathcal{J}_2 \mathcal{S}_3$ \newline
		(v) $u_\ell < u_r,\ \mathtt{f}_1(u_\ell) > \mathtt{f}_2(u_\ell),\ \mathtt{f}_1(u_r) < \mathtt{f}_2(u_r):\, \mathcal{R}_1 \mathcal{J}_2 \mathcal{R}_3$ \newline
		(vi) $u_\ell \leq u_r,\ \mathtt{f}_1(u_\ell) < \mathtt{f}_2(u_\ell):\, \mathcal{S}_1 \mathcal{J}_2 \mathcal{R}_3$
		\\
		\hline
	\end{tabularx}
	\caption{Classification of solutions based on the position of the interaction point}\label{t1}
\end{table}
Hence, the elementary curves associated with three characteristic fields divide the whole region into four constant states (see, Figure \ref{m100}). Further, depending on the position of the interaction point  $ (\rho_{k}, u_{k}) $ and the relations between $ q_{\ell} $ and $ q_{r}, $ $ \mathtt{f}_{1}(u_{\ell}) $ and $ \mathtt{f}_{2}(u_{\ell}), $ $ \mathtt{f}_{1}(u_{r}) $ and $ \mathtt{f}_{2}(u_{r}), $ the Riemann solution's structures are composed of $ \mathcal{S}_{1} \mathcal{J}_{2} \mathcal{S}_{3}, $ $ \mathcal{S}_{1} \mathcal{J}_{2} \mathcal{R}_{3}, $ $ \mathcal{R}_{1} \mathcal{J}_{2} \mathcal{S}_{3}, $ and $ \mathcal{R}_{1} \mathcal{J}_{2} \mathcal{R}_{3} $ (see, Figures \ref{m125}-\ref{m128}). The pictorial representations for the case $ (\rho_{k}, u_{k})\in \mathcal{S}_{q}^{\ell}$, $ (\rho_{k}, u_{k})\in \mathcal{S}_{q}^{r}$ are given; remaining cases can be dealt similarly. 
The classification of Riemann solutions is given in the Table \ref{t1}.

\section{Limiting behaviors of the Riemann solutions for $ u_{\ell}>u_{r}$}
Here, we discuss the behaviors of the Riemann solutions for $ u_{\ell}>u_{r} $ as pressure vanishes, \textit{i.e.,} $ A, B \to 0.  $ It may be noticed that $ 	u=\mathtt{g}_{1}(\rho) \to u_{\ell} $ and $ 	u=\mathtt{g}_{2}(\rho) \to u_{r} $ whenever $ A, B \to 0. $ Thus, for $ u_{\ell}>u_{r},$ there exists  $ m_{o} $ such that $ (\rho_{AB}, u_{AB}) \in  \mathcal{S}_{q}^{\ell}$ and $ (\rho_{AB}, u_{AB}) \in  \mathcal{S}_{q}^{r}$ for any $ A, B \in (0, m_{o}), $ where $ (\rho_{AB}, u_{AB})$ is the point of intersection between $ \mathscr{W}_{q}^{\ell}$ and $ \mathscr{W}_{q}^{r}.$ Thus, we only need to consider the case \bm{$\mathcal{A}$}.
	\begin{theorem}\label{m101} The density coordinate of the point of intersection between the elementary curves $\mathcal{S}_{q}^{\ell}$ and $\mathcal{S}_{q}^{r}$ approaches to infinity as the parameters of the equation of state \eqref{m2} tend to zero, \textit{i.e.,} $\lim_{A, B \to 0} \rho_{AB} = +\infty.$
%		\begin{equation*}\label{m28}
%			\lim_{A, B \to 0} \rho_{AB} = +\infty.
%		\end{equation*}
	\begin{proof}
		In view of  $ (\rho_{AB}, u_{AB}) \in  \mathcal{S}_{q}^{\ell}$ and $ (\rho_{AB}, u_{AB}) \in  \mathcal{S}_{q}^{r},$ we have 
		\begin{equation}\label{m29}
			u_{AB}=u_{\ell}-\sqrt{\mathcal{F}(\rho_{\ell}, \rho_{AB})}, \qquad \rho_{AB}>\rho_{\ell},
		\end{equation}
%		\begin{equation}\label{m29}
%		u_{AB}=u_{\ell}-\sqrt{\left(A(\rho_{AB}^{n}-\rho_{\ell}^{n})+B\left(\frac{1}{\rho_{\ell}^{\alpha}}-\frac{1}{\rho_{AB}^{\alpha}}\right)\right)\left(\frac{1}{\rho_{\ell}}-\frac{1}{\rho_{AB}}\right)}, \qquad \rho_{AB}>\rho_{\ell},
%		\end{equation}
%	and \begin{equation}\label{m30}
%		u_{AB}=u_{r}+\sqrt{\left(A(\rho_{AB}^{n}-\rho_{r}^{n})+B\left(\frac{1}{\rho_{r}^{\alpha}}-\frac{1}{\rho_{AB}^{\alpha}}\right)\right)\left(\frac{1}{\rho_{r}}-\frac{1}{\rho_{AB}}\right)}, \qquad \rho_{AB}>\rho_{r}.
%	\end{equation}
 \begin{equation}\label{m30}
	u_{AB}=u_{r}+\sqrt{\mathcal{F}(\rho_{AB}, \rho_{r})}, \qquad \rho_{AB}>\rho_{r}.
\end{equation}where
	$\mathcal{F}(\rho_{a}, \rho_{b})=\left(A(\rho_{b}^{n}-\rho_{a}^{n})+B\left(\frac{1}{\rho_{a}^{\alpha}}-\frac{1}{\rho_{b}^{\alpha}}\right)\right)\left(\frac{1}{\rho_{a}}-\frac{1}{\rho_{b}}\right).$ Equations \eqref{m29} and \eqref{m30} imply that\vspace{-0.2cm}
%\begin{eqnarray}\label{m36}
%	u_{\ell}-u_{r}=\hspace{-0.5cm}&&\sqrt{\left(A(\rho_{AB}^{n}-\rho_{\ell}^{n})+B\left(\frac{1}{\rho_{\ell}^{\alpha}}-\frac{1}{\rho_{AB}^{\alpha}}\right)\right)\left(\frac{1}{\rho_{\ell}}-\frac{1}{\rho_{AB}}\right)}\nonumber\\&&+\sqrt{\left(A(\rho_{AB}^{n}-\rho_{r}^{n})+B\left(\frac{1}{\rho_{r}^{\alpha}}-\frac{1}{\rho_{AB}^{\alpha}}\right)\right)\left(\frac{1}{\rho_{r}}-\frac{1}{\rho_{AB}}\right)}:= \mathcal{K}(\rho_{AB}, A, B).
%\end{eqnarray}
\begin{eqnarray}\label{m36}
	u_{\ell}-u_{r}=\sqrt{\mathcal{F}(\rho_{\ell}, \rho_{AB})}+\sqrt{\mathcal{F}(\rho_{AB}, \rho_{r})}:= \mathcal{K}(\rho_{AB}, A, B).
\end{eqnarray} Choose $ A_{1}, B_{1}, A_{2}, B_{2} \in (0, m_{o})$ such that $ A_{1}>A_{2}$ and $ B_{1}>B_{2}. $ Let us consider $ \rho_{A_{1}B_{1}}> \rho_{A_{2}B_{2}}. $ Then \eqref{m36} yields that $ \mathcal{K}(\rho_{A_{1}B_{1}}, A_{1}, B_{1}) >  \mathcal{K}(\rho_{A_{2}B_{2}}, A_{2}, B_{2}),$ which, in turn,  contradicts with $ \mathcal{K}(\rho_{A_{1}B_{1}}, A_{1}, B_{1}) =  \mathcal{K}(\rho_{A_{2}B_{2}}, A_{2}, B_{2})= u_{\ell}-u_{r}.$ Therefore, $ \rho_{AB} $ is monotonically decreasing \wrt {} $ A, B.$
The monotonicity of $ \rho_{AB}$ implies either $ \lim_{A, B \to 0} \rho_{AB} = +\infty $ or $ 	\lim_{A, B \to 0} \rho_{AB} =\text{constant} \neq +\infty. $ If $ \lim_{A, B \to 0} \rho_{AB} =\text{constant} \neq +\infty, $ then \eqref{m36} infers the contradiction to $ u_{\ell}>u_{r}.$ Hence, $ 	\lim_{A, B \to 0} \rho_{AB} = +\infty.$
	\end{proof}
	\end{theorem}We divide this discussion into the following two cases depending on the relations between $ q_{\ell} $ and $ q_{r},$ \textit{i.e.,} 
	\vspace*{-.4cm}
	\subsection{Case (i)  $\boldsymbol{q_{\ell}>q_{r}}$} In view of  \eqref{m40}, one has
\begin{equation}\label{m41}
	\mathtt{f}_{1}(u_{\ell})=\mathtt{h}_{1}(\mathtt{g}_{1}^{-1}(u_{\ell}))=\mathcal{P}(\mathtt{g}_{1}^{-1}(u_{\ell}))+q_{\ell}=\mathcal{P}(\rho_{\ell})+q_{\ell}\to q_{\ell} \quad \text{as} \quad A, B \to 0.
\end{equation}
and \begin{equation}\label{m42}
		 \mathtt{f}_{2}(u_{\ell})=\mathtt{h}_{2}(\mathtt{g}_{2}^{-1}(u_{\ell}))=\mathcal{P}(\mathtt{g}_{2}^{-1}(u_{\ell}))+q_{r} =\mathcal{P}(\rho_{d})+q_{r}\to \rho_{r}(u_{\ell}-u_{r})^2 +q_{r} \quad \text{as} \quad A, B \to 0,
	\end{equation}where $\rho_{d}$ satisfies
	\begin{equation}\label{m43}
		u_{\ell}=u_{r}+\sqrt{\mathcal{F}(\rho_{d}, \rho_{r})}, \qquad \rho_{d}>\rho_{r}.
	\end{equation}
%\begin{equation}\label{m43}
%	u_{\ell}=u_{r}+\sqrt{\left(A(\rho_{d}^{n}-\rho_{r}^{n})+B\left(\frac{1}{\rho_{r}^{\alpha}}-\frac{1}{\rho_{d}^{\alpha}}\right)\right)\left(\frac{1}{\rho_{r}}-\frac{1}{\rho_{d}}\right)}, \qquad \rho_{d}>\rho_{r}.
%\end{equation}
This discussion will again be divided into the following three sub-cases:
\\\textbf{(a).} $ \boldsymbol{ q_{\ell} < \rho_{r}(u_{\ell}-u_{r})^2 +q_{r}}:$ In this case, there exists $ m_{1}\in (0, m_{o}) $ such that $ \mathtt{f}_{1}(u_{\ell})< \mathtt{f}_{2}(u_{\ell})$ for $ A, B \in (0, m_{1}). $ Now, for fixed $ A, B \in (0, m_{1}), $ the Riemann solution is composed by
% $ \mathcal{S}_{1} \mathcal{J}_{2} \mathcal{S}_{3},$ 
 back shock wave, contact discontinuity and front shock wave, \textit{i.e.,}\vspace{-0.15cm}
\begin{equation}\label{m44}
	(\rho^{AB}, q^{AB},  u^{AB})(x, t) =	\begin{cases}
		(\rho_{\ell}, q_{\ell}, u_{\ell}), \quad \qquad\qquad -\infty<x/t<\sigma_{1}^{AB},
		\\ (\rho_{*_{1AB}}, q_{*_{1AB}}, u_{*_{1AB}}), \quad \sigma_{1}^{AB}<x/t<\sigma_{2}^{AB},
		\\ (\rho_{*_{2AB}}, q_{*_{2AB}}, u_{*_{2AB}}), \quad \sigma_{2}^{AB}<x/t<\sigma_{3}^{AB},
		\\ (\rho_{r}, q_{r}, u_{r}), \qquad \qquad \quad \,\, \sigma_{3}^{AB}<x/t<+\infty.
	\end{cases}
\end{equation}
Here, $ (\rho_{\ell}, q_{\ell}, u_{\ell}) $ and $ (\rho_{*_{1AB}}, q_{*_{1AB}}, u_{*_{1AB}}) $ are connected with the following back shock wave ($\mathcal{S}_{1}$):\vspace{-0.15cm} 
\begin{equation}\label{m45}
	\mathcal{S}_{1}=\begin{cases}
		\sigma_{1}^{AB}=u_{*_{1AB}}-\sqrt{\frac{\rho_{\ell}(\mathcal{P}(\rho_{*_{1AB}})-\mathcal{P}(\rho_{\ell}))}{\rho_{*_{1AB}}(\rho_{*_{1AB}}-\rho_{\ell})}}=u_{\ell}-\sqrt{\frac{\rho_{*_{1AB}}(\mathcal{P}(\rho_{*_{1AB}})-\mathcal{P}(\rho_{\ell}))}{\rho_{\ell}(\rho_{*_{1AB}}-\rho_{\ell})}},\\u_{*_{1AB}}=u_{\ell}-\sqrt{(\mathcal{P}(\rho_{*_{1AB}})-\mathcal{P}(\rho_{\ell}))\left(\frac{1}{\rho_{\ell}}-\frac{1}{\rho_{*_{1AB}}}\right)}, \\q_{\ell}=q_{*_{1AB}},\quad \rho_{\ell}<\rho_{*_{1AB}}<\rho_{AB},
	\end{cases}
\end{equation}$ (\rho_{*_{1AB}}, q_{*_{1AB}}, u_{*_{1AB}}) $ and $ (\rho_{*_{2AB}}, q_{*_{2AB}}, u_{*_{2AB}}) $ are connected with contact discontinuity ($\mathcal{J}_{2}$): \vspace{-0.1cm}
\begin{equation}\label{m46}
	\mathcal{J}_{2}=\begin{cases}
		\sigma_{2}^{AB}=u_{*_{1AB}}=u_{*_{2AB}}\\q_{*_{1AB}}+\mathcal{P}(\rho_{*_{1AB}})=q_{*_{2AB}}+\mathcal{P}(\rho_{*_{2AB}}),
	\end{cases}
\end{equation} and $ (\rho_{*_{2AB}}, q_{*_{2AB}}, u_{*_{2AB}}) $ and $ (\rho_{r}, q_{r}, u_{r}) $ are connected with front shock wave ($\mathcal{S}_{3}$):
\begin{equation}\label{m47}
	\mathcal{S}_{3}=\begin{cases}
		\sigma_{3}^{AB}=u_{r}+\sqrt{\frac{\rho_{*_{2AB}}(\mathcal{P}(\rho_{r})-\mathcal{P}(\rho_{*_{2AB}}))}{\rho_{r}(\rho_{r}-\rho_{*_{2AB}})}}=u_{*_{2AB}}+\sqrt{\frac{\rho_{r}(\mathcal{P}(\rho_{r})-\mathcal{P}(\rho_{*_{2AB}}))}{\rho_{*_{2AB}}(\rho_{r}-\rho_{*_{2AB}})}},\\u_{*_{2AB}}=u_{r}+\sqrt{(\mathcal{P}(\rho_{*_{2AB}})-\mathcal{P}(\rho_{r}))\left(\frac{1}{\rho_{r}}-\frac{1}{\rho_{*_{2AB}}}\right)}, \\q_{*_{2AB}}=q_{r},\quad \rho_{*_{2AB}}>\rho_{AB}>\rho_{r}.
	\end{cases}
\end{equation}  

\begin{theorem}\label{m103} The intermediate density between the contact discontinuity and the front shock wave blows up as the parameters in the equation of state \eqref{m2} get closer to zero, \textit{i.e.,} $\lim_{A, B \to 0} \rho_{*_{2AB}} = +\infty.$
%\begin{equation*}\label{m48} 
%	\lim_{A, B \to 0} \rho_{*_{2AB}} = +\infty.
%\end{equation*}
\begin{proof}
In view of theorem \ref{m101} and \eqref{m47},  we get the required result.
\end{proof}
\end{theorem}

\begin{theorem}\label{m104}
	The intermediate density between the back shock wave and the contact discontinuity tends to an extreme concentration as pressure drops to zero, \textit{i.e.,} $\lim_{A, B \to 0} \rho_{*_{1AB}} = +\infty.$
%	\begin{equation*}\label{m49}
%		\lim_{A, B \to 0} \rho_{*_{1AB}} = +\infty.
%	\end{equation*}
\begin{proof}
	From $ \eqref{m45}_{2}, $ $ \eqref{m46}_{2}, $ and $ \eqref{m47}_{2}, $ we have
		\begin{eqnarray}\label{m50}
		u_{\ell}-u_{r}=\sqrt{\mathcal{F}(\rho_{\ell}, \rho_{*_{1AB}})}+\sqrt{\left\{\mathcal{P}( \rho_{*_{1AB}})-\mathcal{P}(\rho_{r})+q_{\ell}-q_{r}\right\}\left(\frac{1}{\rho_{r}}-\frac{1}{\rho_{*_{2AB}}}\right)}.
	\end{eqnarray}
%	\begin{eqnarray}\label{m50}
%		u_{\ell}-u_{r}=\hspace{-0.5cm}&&\sqrt{\left(A(\rho_{*_{1AB}}^{n}-\rho_{\ell}^{n})+B\left(\frac{1}{\rho_{\ell}^{\alpha}}-\frac{1}{\rho_{*_{1AB}}^{\alpha}}\right)\right)\left(\frac{1}{\rho_{\ell}}-\frac{1}{\rho_{*_{1AB}}}\right)}\nonumber\\&&+\sqrt{\left(A(\rho_{*_{1AB}}^{n}-\rho_{r}^{n})+B\left(\frac{1}{\rho_{r}^{\alpha}}-\frac{1}{\rho_{*_{1AB}}^{\alpha}}\right)+q_{\ell}-q_{r}\right)\left(\frac{1}{\rho_{r}}-\frac{1}{\rho_{*_{2AB}}}\right)}.
%	\end{eqnarray}
 On employing $ \lim_{A, B \to 0} $ in $ \eqref{m50}, $ under  the consideration $ 	\lim_{A, B \to 0} \rho_{*_{1AB}}=L \in (\rho_{\ell}, +\infty),$ one  gets $ q_{\ell}=\rho_{r}(u_{\ell}-u_{r})^2+q_{r}, $ which contradicts  $ q_{\ell}<\rho_{r}(u_{\ell}-u_{r})^2+q_{r}. $ 
\end{proof}
\end{theorem}

\begin{theorem}\label{m105} As the parameters in the equation of state \eqref{m2} approach $0,$ the intermediate velocities between the back shock wave and contact discontinuity, as well as between the contact discontinuity and front shock wave, remain finite. Additionally, one has the following
%	\begin{equation}\label{m51}
%	\lim_{A, B \to 0} u_{*_{1AB}} = \lim_{A, B \to 0} u_{*_{2AB}} = \sigma,
%	\end{equation}
\begin{equation}\label{m52}
\lim_{A, B \to 0} u_{*_{1AB}} = \lim_{A, B \to 0} u_{*_{2AB}}=	\sigma=\begin{cases}
		\frac{u_{\ell}+u_{r}}{2}+\frac{q_{\ell}-q_{r}}{2\rho_{r}(u_{\ell}-u_{r})}, \qquad \,\,\,\qquad \qquad \qquad \,\,  \rho_{\ell}=\rho_{r},\\ \frac{(\rho_{\ell}u_{\ell}-\rho_{r}u_{r})-\sqrt{\rho_{\ell}\rho_{r}(u_{\ell}-u_{r})^2-(\rho_{\ell}-\rho_{r})(q_{\ell}-q_{r})}}{\rho_{\ell}-\rho_{r}}, \,\,  \rho_{\ell}\neq\rho_{r}.
	\end{cases}
\end{equation}
\begin{proof}
If we assume that $ \lim_{A, B \to 0} A\rho_{*_{1AB}}^n=M, $ then $ \eqref{m50} $ gives
\begin{equation}\label{m53}
	u_{\ell}-\sqrt{\frac{M}{\rho_{\ell}}}=u_{r}+\sqrt{\frac{M+q_{\ell}-q_{r}}{\rho_{r}}}:=\sigma,
\end{equation}
which, in turn, implies that
\begin{equation}\label{m54}
	M=\rho_{\ell}(u_{\ell}-\sigma)^2=\rho_{r}(\sigma-u_{r})^2-(q_{\ell}-q_{r}).
\end{equation}On simplifying \eqref{m54} and  taking $ \lim_{A, B \to 0}$ in $ \eqref{m45}_{2} $ and $ \eqref{m47}_{2}, $ we get the required result.
\end{proof}
\end{theorem}
\begin{theorem}\label{m106} The vanishing limits of the parameters in the equation of state \eqref{m2} cause the back shock wave, contact discontinuity, and front shock wave curves to coincide, \textit{i.e.,} $\lim_{A, B \to 0}\sigma_{1}^{AB}=\lim_{A, B \to 0}\sigma_{2}^{AB}=\lim_{A, B \to 0}\sigma_{3}^{AB}=\sigma.$
%	\begin{equation*}\label{m55}
%	\lim_{A, B \to 0}\sigma_{1}^{AB}=\lim_{A, B \to 0}\sigma_{2}^{AB}=\lim_{A, B \to 0}\sigma_{3}^{AB}=\sigma.
%	\end{equation*}
\begin{proof}
	The result can be obtained by employing  $ \lim_{A, B \to 0} $ in $ \eqref{m45}_{1}, $ $ \eqref{m46}_{1}, $ and $ \eqref{m47}_{1}. $
\end{proof}
\end{theorem}
%\begin{remark}
%	Theorems \ref{m103}-\ref{m106} imply that the intermediate densities $ \rho_{*_{1AB}} $ and $ \rho_{*_{2AB}} $ blow up as pressure drops to zero; indeed, $ 1 $-shock wave ($\mathcal{S}_{1}$), contact discontinuity ($\mathcal{J}_{2}$), and $ 3 $-shock wave ($\mathcal{S}_{3}$) coincide with each other as $ A, B \to 0. $\vspace*{-0.1cm}	
%\end{remark}

\begin{theorem}\label{m107}
	\begin{equation*}\label{m56}
	\lim_{A, B \to 0}	\int_{\sigma_{1}^{AB}}^{\sigma_{3}^{AB}} \rho^{AB} d \xi=-\sigma[\rho]+[\rho u],
	\end{equation*}
\begin{equation*}\label{m57}
	\lim_{A, B \to 0}	\int_{\sigma_{1}^{AB}}^{\sigma_{3}^{AB}} \rho^{AB} q^{AB} d \xi=-\sigma[\rho q]+[\rho q u],
\end{equation*} and
\begin{equation*}\label{m58}
	\lim_{A, B \to 0}	\int_{\sigma_{1}^{AB}}^{\sigma_{3}^{AB}} \rho^{AB} u^{AB} d \xi=(-\sigma[\rho]+[\rho u])\sigma,
\end{equation*} where $ [a]=a_{\ell}-a_{r}. $
\begin{proof}
	We have \small
	\begin{eqnarray}
		\lim_{A, B \to 0}	\int_{\sigma_{1}^{AB}}^{\sigma_{3}^{AB}} \rho^{AB} d \xi=\hspace{-0.7cm}&& \lim_{A, B \to 0} \rho_{*_{1AB}}(\sigma_{2}^{AB}-\sigma_{1}^{AB})+\lim_{A, B \to 0} \rho_{*_{2AB}}(\sigma_{3}^{AB}-\sigma_{2}^{AB}), \nonumber \\&&= \lim_{A, B \to 0} \rho_{*_{1AB}}\sqrt{\frac{\rho_{\ell}(\mathcal{P}(\rho_{*_{1AB}})-\mathcal{P}(\rho_{\ell}))}{\rho_{*_{1AB}}(\rho_{*_{1AB}}-\rho_{\ell})}}+\lim_{A, B \to 0} \rho_{*_{2AB}}\sqrt{\frac{\rho_{r}(\mathcal{P}(\rho_{r})-\mathcal{P}(\rho_{*_{2AB}}))}{\rho_{*_{2AB}}(\rho_{r}-\rho_{*_{2AB}})}}, \nonumber\\ && = \sqrt{\rho_{\ell} M} + \sqrt{\rho_{r} (M+q_{\ell}-q_{r})} =\rho_{\ell}(u_{\ell}-\sigma)+\rho_{r}(\sigma-u_{r})=-\sigma[\rho]+[\rho u], \nonumber
	\end{eqnarray}
\begin{equation*}\label{m59}
	\lim_{A, B \to 0}	\int_{\sigma_{1}^{AB}}^{\sigma_{3}^{AB}} \rho^{AB} q^{AB} d \xi=\lim_{A, B \to 0}\rho_{*_{1AB}}q_{\ell}(\sigma_{2}^{AB}-\sigma_{1}^{AB})+\lim_{A, B \to 0} \rho_{*_{2AB}}q_{r}(\sigma_{3}^{AB}-\sigma_{2}^{AB})=-\sigma[\rho q]+[\rho q u],
\end{equation*}
and \begin{equation*}\label{m60}
	\lim_{A, B \to 0}	\int_{\sigma_{1}^{AB}}^{\sigma_{3}^{AB}} \rho^{AB} u^{AB} d \xi=\lim_{A, B \to 0}\rho_{*_{1AB}}u_{*_{1AB}}(\sigma_{2}^{AB}-\sigma_{1}^{AB})+\lim_{A, B \to 0} \rho_{*_{2AB}}u_{*_{2AB}}(\sigma_{3}^{AB}-\sigma_{2}^{AB})=(-\sigma[\rho]+[\rho u])\sigma.
\end{equation*} 
\end{proof}
\end{theorem}

\begin{note}
	\normalsize
	Denote \begin{equation*}
		H_{a}(x-\sigma t)=\begin{cases}
			a_{\ell}, \qquad x<\sigma t,\\
			a_{r}, \qquad x>\sigma t.
		\end{cases}
	\end{equation*}
\end{note} Now, we have the following theorem characterizing the vanishing-parameters limits of $ \rho^{AB}, $ $ \rho^{AB}q^{AB}, $ and $ \rho^{AB}u^{AB} .$
\begin{theorem}\label{123}
	Assume $ u_{\ell}>u_{r}, $ $ q_{\ell}>q_{r}, $ $ q_{\ell} < \rho_{r}(u_{\ell}-u_{r})^2 +q_{r},$ and for all fixed $A, B\in(0, m_{1}),$  $ (\rho^{AB}, q^{AB},  u^{AB}) $ is the $ \mathcal{S}_{1} \mathcal{J}_{2} \mathcal{S}_{3} $ solution of the Riemann problem \eqref{m1}-\eqref{m2} as constructed in section $2.$ Then, for $ A, B \to 0, $ $ \rho^{AB}, $ $ \rho^{AB}q^{AB}, $ and $ \rho^{AB}u^{AB} $ converge in the distributional sense and their limit functions can be expressed as
	\begin{equation*}\label{m61}
		H_{\rho}(x-\sigma t)+\left(-\sigma[\rho]+[\rho u]\right)t \delta(x-\sigma t),
	\end{equation*}
	\begin{equation*}\label{m62}
		H_{\rho q}(x-\sigma t)+\left(-\sigma[\rho q]+[\rho q u]\right)t \delta(x-\sigma t),
	\end{equation*}and
	\begin{equation*}\label{m63}
		H_{\rho u}(x-\sigma t)+\left(-\sigma[\rho]+[\rho u]\right)\sigma t \delta(x-\sigma t),
	\end{equation*}
	respectively.
	\begin{proof}
	 The proof can be done similar to the proof of theorem $3.7$ in reference \cite{cheng2018}. So, we omit the proof. 
	\end{proof} 
\end{theorem}
\begin{remark}
	Theorems \ref{m103}-\ref{123} yield that the state variable $\rho$ possesses a  singularity as a weighted Dirac delta function whenever the parameters in the equation of state drop to zero. Indeed, in the vanishing limits of the parameters of equation of state, back shock, contact discontinuity and front shock waves coincide  to form a new type of nonlinear hyperbolic wave, namely, delta shock wave on which the state variable $\rho$ contains the Dirac delta function.
	% Thus, the delta shock wave with the Dirac delta function in the state variable $\rho$ will appear in the solution. Hence, in the vanishing-parameters limiting case, the Riemann solution composed by $ 1 $-shock wave, contact discontinuity, and $ 3 $-shock wave converges to a single delta shock solution which is same as the
	%corresponding Riemann solution of the pressureless Cargo-LeRoux model.
\end{remark}
\noindent 
\\\textbf{(b).} $ \boldsymbol{ q_{\ell} > \rho_{r}(u_{\ell}-u_{r})^2 +q_{r}}:$ In this case, there exists $ m_{2}\in (0, m_{o}) $ such that $ \mathtt{f}_{1}(u_{\ell})> \mathtt{f}_{2}(u_{\ell})$ for $ A, B \in (0, m_{2}). $ Now, for fixed $ A, B \in (0, m_{2}), $ the Riemann solution is composed by
 %$ \mathcal{R}_{1} \mathcal{J}_{2} \mathcal{S}_{3},$ 
 back rarefaction wave, contact discontinuity and front shock wave, \textit{i.e.,}
\begin{equation}\label{m83}
	(\rho^{AB}, q^{AB},  u^{AB})(x, t) =	\begin{cases}
		(\rho_{\ell}, q_{\ell}, u_{\ell}), \quad \qquad\qquad -\infty<x/t<u_{\ell}-\sqrt{\mathcal{P}^\prime(\rho_{\ell})},\\
	\qquad	\mathcal{R}_{1}, \qquad\qquad\qquad\,\, u_{\ell}-\sqrt{\mathcal{P}^\prime(\rho_{\ell})}<x/t<u_{*_{1AB}}-\sqrt{\mathcal{P}^\prime(\rho_{*_{1AB}})},
		\\ (\rho_{*_{1AB}}, q_{*_{1AB}}, u_{*_{1AB}}), \quad u_{*_{1AB}}-\sqrt{\mathcal{P}^\prime(\rho_{*_{1AB}})}<x/t<\sigma_{2}^{AB},
		\\ (\rho_{*_{2AB}}, q_{*_{2AB}}, u_{*_{2AB}}), \quad \sigma_{2}^{AB}<x/t<\sigma_{3}^{AB},
		\\ (\rho_{r}, q_{r}, u_{r}), \qquad \qquad \quad \,\, \sigma_{3}^{AB}<x/t<+\infty.
	\end{cases}
\end{equation}
Here, $ (\rho_{\ell}, q_{\ell}, u_{\ell}) $ and $ (\rho_{*_{1AB}}, q_{*_{1AB}}, u_{*_{1AB}}) $ are connected with back rarefaction wave ($\mathcal{R}_{1}$):
\begin{equation}\label{m84}
	\mathcal{R}_{1}=\begin{cases}
		x/t=u-\sqrt{\mathcal{P}^\prime(\rho)},\\q_{\ell}=q_{ab},\\u_{ab}=u_{\ell}-\int_{\rho_{\ell}}^{\rho_{ab}}\frac{\sqrt{\mathcal{P}^\prime(s)}}{s}ds, \quad \rho_{*_{1AB}}\leq\rho_{ab}\leq\rho_{\ell}<\rho_{AB},
	\end{cases}
\end{equation}
$ (\rho_{*_{1AB}}, q_{*_{1AB}}, u_{*_{1AB}}) $ and $ (\rho_{*_{2AB}}, q_{*_{2AB}}, u_{*_{2AB}}) $ are connected with contact discontinuity ($\mathcal{J}_{2}$)
%\begin{equation}\label{m85}
%	\mathcal{J}_{2}=\begin{cases}
%		\sigma_{2}^{AB}=u_{*_{1AB}}=u_{*_{2AB}},\\q_{*_{1AB}}+\mathcal{P}(\rho_{*_{1AB}})=q_{*_{2AB}}+\mathcal{P}(\rho_{*_{2AB}}),
%	\end{cases}
%\end{equation}
and $ (\rho_{*_{2AB}}, q_{*_{2AB}}, u_{*_{2AB}}) $ and $ (\rho_{r}, q_{r}, u_{r}) $ are connected with front shock wave ($\mathcal{S}_{3}$). The contact discontinuity ($\mathcal{J}_{2}$) and front shock wave ($\mathcal{S}_{3}$) are represented by the above expressions \eqref{m46} and \eqref{m47}, respectively. Similar to the proof done in theorem \ref{m104}, one can achieve $\lim_{A, B \to 0} \rho_{*_{2AB}} = +\infty.$ 
%\begin{equation}\label{m86}
%	\mathcal{S}_{3}=\begin{cases}
%		\sigma_{3}^{AB}=u_{r}+\sqrt{\frac{\rho_{*_{2AB}}(\mathcal{P}(\rho_{r})-\mathcal{P}(\rho_{*_{2AB}}))}{\rho_{r}(\rho_{r}-\rho_{*_{2AB}})}}=u_{*_{2AB}}+\sqrt{\frac{\rho_{r}(\mathcal{P}(\rho_{r})-\mathcal{P}(\rho_{*_{2AB}}))}{\rho_{*_{2AB}}(\rho_{r}-\rho_{*_{2AB}})}},\\u_{*_{2AB}}=u_{r}+\sqrt{(\mathcal{P}(\rho_{*_{2AB}})-\mathcal{P}(\rho_{r}))\left(\frac{1}{\rho_{r}}-\frac{1}{\rho_{*_{2AB}}}\right)}, \\q_{*_{2AB}}=q_{r},\quad \rho_{*_{2AB}}>\rho_{AB}>\rho_{r}.
%	\end{cases}
%\end{equation}  

%\begin{theorem}\label{m87}
%	\begin{equation*}\label{m88}
%		\lim_{A, B \to 0} \rho_{*_{2AB}} = +\infty.
%	\end{equation*}
%	\begin{proof} The proof can be done by
%	using theorem \ref{m101} and the relation $ \rho_{*_{2AB}}>\rho_{AB}.$ 
%	\end{proof}
%\end{theorem}

\begin{theorem}\label{m89} The intermediate state between the back rarefaction wave and contact discontinuity turns out to be a vacuum state as the parameters in equation of state \eqref{m2} converge to zero, \textit{i.e.,} 	$\lim_{A, B \to 0} \rho_{*_{1AB}} = 0.$
%	\begin{equation*}\label{m90}
%		\lim_{A, B \to 0} \rho_{*_{1AB}} = 0.
%	\end{equation*}
	\begin{proof}
		From $ \eqref{m46}_{2}, $ $ \eqref{m47}_{3}, $ and $ \eqref{m84}_{2}, $ we have
%		\begin{eqnarray}\label{m91}
%			u_{\ell}-u_{r}=\hspace{-0.5cm}&&\sqrt{\left(A(\rho_{*_{1AB}}^{n}-\rho_{r}^{n})+B\left(\frac{1}{\rho_{r}^{\alpha}}-\frac{1}{\rho_{*_{1AB}}^{\alpha}}\right)+q_{\ell}-q_{r}\right)\left(\frac{1}{\rho_{r}}-\frac{1}{\rho_{*_{2AB}}}\right)}\nonumber\\&& \hspace{8.5cm}+\int_{\rho_{\ell}}^{\rho_{*_{1AB}}}\frac{\sqrt{\mathcal{P}^\prime(s)}}{s}ds.
%		\end{eqnarray} 
\begin{eqnarray}\label{m91}
	u_{\ell}-u_{r}=\sqrt{\left\{\mathcal{P}( \rho_{*_{1AB}})-\mathcal{P}(\rho_{r})+q_{\ell}-q_{r}\right\}\left(\frac{1}{\rho_{r}}-\frac{1}{\rho_{*_{2AB}}}\right)}+\int_{\rho_{\ell}}^{\rho_{*_{1AB}}}\frac{\sqrt{\mathcal{P}^\prime(s)}}{s}ds.
\end{eqnarray} Taking $ \lim_{A, B \to 0} $ in $ \eqref{m91}, $ with the assumption  $ 	\lim_{A, B \to 0} \rho_{*_{1AB}}=L (\neq 0) \in (0, \rho_{\ell}), $ one gets $ q_{\ell}=\rho_{r}(u_{\ell}-u_{r})^2+q_{r},$ which contradicts with $ q_{\ell}>\rho_{r}(u_{\ell}-u_{r})^2+q_{r}. $ Thus, $ \lim_{A, B \to 0} \rho_{*_{1AB}} = 0.$
	\end{proof}
\end{theorem}
\begin{theorem}\label{m90} Set $\lim_{A, B \to 0}u_{*_{1AB}}=\nu.$ Then
	\begin{equation*}\label{m92}
\lim_{A, B \to 0} u_{*_{1AB}}=\lim_{A, B \to 0} u_{*_{2AB}}=\lim_{A, B \to 0}\sigma_{2}^{AB}=\lim_{A, B \to 0}\sigma_{3}^{AB}=\nu \in \left[u_{r},  u_{r}+\sqrt{\frac{q_{\ell}-q_{r}}{\rho_{r}}}\,\,\right].		
	\end{equation*}
\begin{proof}Now,
	$\rho_{*_{1AB}}<\rho_{\ell}$ and $ q_{*_{1AB}}+\mathcal{P}(\rho_{*_{1AB}})=q_{*_{2AB}}+\mathcal{P}(\rho_{*_{2AB}}) $ imply that
	\begin{equation}\label{m93}
		A\rho_{r}^{n}-\frac{B}{\rho_{r}^{\alpha}}+q_{r}-q_{\ell}<A\rho_{*_{1AB}}^n-\frac{B}{\rho_{*_{1AB}}^{\alpha}}<A\rho_{\ell}^{n}-\frac{B}{\rho_{\ell}^{\alpha}}.
	\end{equation}Taking $\lim_{A, B \to 0},$ we have
\begin{equation}\label{m94}
	q_{r}-q_{\ell} \leq\lim_{A, B \to 0}\mathcal{P}(\rho_{*_{1AB}})\leq0.
\end{equation} Using \eqref{m94} in $ \eqref{m46}_{2}, $ one can obtain \begin{equation}\label{m95}
0\leq\lim_{A, B \to 0}\mathcal{P}(\rho_{*_{2AB}})\leq q_{\ell}-q_{r}.
\end{equation}In view of  \eqref{m46}, \eqref{m47}, and \eqref{m95}, we have the desired result.
\end{proof}
\end{theorem}
\begin{theorem}\label{m90b}
	\begin{equation*}\label{m96}
		\lim_{A, B \to 0}	\int_{\sigma_{2}^{AB}}^{\sigma_{3}^{AB}} \rho^{AB} d \xi=\rho_{r}(\nu-u_{r}),
	\end{equation*}
	\begin{equation*}\label{m97}
		\lim_{A, B \to 0}	\int_{\sigma_{2}^{AB}}^{\sigma_{3}^{AB}} \rho^{AB} q^{AB} d \xi=\rho_{r}q_{r}(\nu-u_{r}),
	\end{equation*} and
	\begin{equation*}\label{m98}
		\lim_{A, B \to 0}	\int_{\sigma_{2}^{AB}}^{\sigma_{3}^{AB}} \rho^{AB} u^{AB} d \xi=\rho_{r}u_{r}(\nu-u_{r})+q_{\ell}-q_{r}+\mathcal{G},
	\end{equation*} where $ \mathcal{G}=\lim_{A, B \to 0}\mathcal{P}(\rho_{*_{1AB}}). $
\begin{proof}
	From the Rankine Hugoniot relation $ \eqref{m14}_{1} $ for $\mathcal{J}_{2}$ and $\mathcal{S}_{3},$ we have 
	\begin{equation}\label{99}
	\begin{cases}
			-\sigma_{3}^{AB}(\rho_{*_{2AB}}-\rho_{r})+(\rho_{*_{2AB}} u_{*_{2AB}}-\rho_{r}u_{r})=0,\\
			-\sigma_{2}^{AB}(\rho_{*_{1AB}}-\rho_{*_{2AB}})+(\rho_{*_{1AB}} u_{*_{1AB}}-\rho_{*_{2AB}} u_{*_{2AB}})=0,
	\end{cases}
	\end{equation} which, in turn, imply that
	\begin{equation*}\label{m110}
	\lim_{A, B \to 0}(\sigma_{3}^{AB}-\sigma_{2}^{AB})\rho_{*_{2AB}}=\lim_{A, B \to 0}\rho_{r}(\sigma_{3}^{AB}-u_{r})=\rho_{r}(\nu-u_{r}).
\end{equation*}Hence	\begin{equation*}\label{m111}
\lim_{A, B \to 0}	\int_{\sigma_{2}^{AB}}^{\sigma_{3}^{AB}} \rho^{AB} d \xi=\lim_{A, B \to 0}(\sigma_{3}^{AB}-\sigma_{2}^{AB})\rho_{*_{2AB}}=\rho_{r}(\nu-u_{r}).
\end{equation*} Similarly, from the Rankine Hugoniot relation $ \eqref{m14}_{2} $ for $\mathcal{J}_{2}$ and $\mathcal{S}_{3},$ we have
	\begin{equation*}\label{m112}
	\lim_{A, B \to 0}	\int_{\sigma_{2}^{AB}}^{\sigma_{3}^{AB}} \rho^{AB} q^{AB} d \xi=\lim_{A, B \to 0}(\sigma_{3}^{AB}-\sigma_{2}^{AB})\rho_{*_{2AB}}q_{*_{2AB}}=\lim_{A, B \to 0}\rho_{r}q_{r}(\sigma_{3}^{AB}-u_{r})=\rho_{r}q_{r}(\nu-u_{r}),
\end{equation*}
and from the Rankine Hugoniot relation $ \eqref{m14}_{3} $ for $\mathcal{J}_{2}$ and $\mathcal{S}_{3},$ one has
	\begin{eqnarray*}\label{m113}
	\lim_{A, B \to 0}	\int_{\sigma_{2}^{AB}}^{\sigma_{3}^{AB}} \rho^{AB} u^{AB} d \xi\hspace{-0.5cm}&&=\lim_{A, B \to 0}(\sigma_{3}^{AB}-\sigma_{2}^{AB})\rho_{*_{2AB}}u_{*_{2AB}},\\&&=\lim_{A, B \to 0}\rho_{r}u_{r}(\sigma_{3}^{AB}-u_{r})+q_{\ell}-q_{r}+\lim_{A, B \to 0}\mathcal{P}(\rho_{*_{1AB}}),\\&&=\rho_{r}u_{r}(\nu-u_{r})+q_{\ell}-q_{r}+\mathcal{G}.
\end{eqnarray*}
\end{proof}	
\end{theorem}
\begin{note}
	Denote \begin{equation*}
	\overline{H}_{a}(x-\nu t)=\begin{cases}
			a_{\ell},\, \qquad x\leq u_{\ell} t,\\0, \qquad \,\,u_{\ell} t<x <\nu t,\\
			a_{r}, \qquad x>\nu t.
		\end{cases}
	\end{equation*}
\end{note}
\begin{theorem}\label{m90a}
	Assume $ u_{\ell}>u_{r}, $ $ q_{\ell}>q_{r}, $ $ q_{\ell} > \rho_{r}(u_{\ell}-u_{r})^2 +q_{r}, $ and for all fixed $ A, B \in (0, m_{2}), $  $ (\rho^{AB}, q^{AB},  u^{AB}) $ is the $ \mathcal{R}_{1} \mathcal{J}_{2} \mathcal{S}_{3} $ solution of the Riemann problem \eqref{m1}-\eqref{m2} as constructed in section $2.$ Then, for $ A, B \to 0, $ $ \rho^{AB}, $ $ \rho^{AB}q^{AB}, $ and $ \rho^{AB}u^{AB} $ converge in the distributional sense and their limit functions can be expressed as
	\begin{equation*}\label{m114}
		\overline{H}_{\rho}(x-\nu t)+\rho_{r}(\nu-u_{r})t \delta(x-\nu t),
	\end{equation*}
	\begin{equation*}\label{m115}
		\overline{H}_{\rho q}(x-\nu t)+\rho_{r}q_{r}(\nu-u_{r})t \delta(x-\nu t),
	\end{equation*}and
	\begin{equation*}\label{m116}
		\overline{H}_{\rho u}(x-\nu t)+(\rho_{r}u_{r}(\nu-u_{r})+q_{\ell}-q_{r}+\mathcal{G})t \delta(x-\nu t),
	\end{equation*}
	respectively.
	\begin{proof}
	The proof of this theorem is similar to that of theorem $3.13$ in reference \cite{cheng2018}. So, we omit the proof.
	\end{proof}
\end{theorem}
\begin{remark}
	Theorem \ref{m89} implies that the non-vacuum intermediate state between the back rarefaction wave and the contact discontinuity tends to a vacuum state as the pressure vanishes. Further, theorems \ref{m90}-\ref{m90a} infer that the intermediate density between the contact discontinuity and the front shock wave converges to a weighted delta measure which forms the delta shock wave. Thus, the Riemann solution composed by back rarefaction wave, contact discontinuity and front shock wave converges to a solution having vacuum on the left hand side and delta shock wave on the right hand side as pressure drops to zero. 
\end{remark}

\noindent\textbf{(c).} $ \boldsymbol{ q_{\ell} = \rho_{r}(u_{\ell}-u_{r})^2 +q_{r}}:$ In order to deal with this case, we perturbed our initial data \eqref{m3} as follows
 \begin{equation} \label{m123}
 	(\rho, q,  u)_{t=0} =	\begin{cases}
 		(\rho_{\ell}, q_{\ell}^{\epsilon}=q_{\ell}+\epsilon, u_{\ell}), \quad x<0,
 		\\ (\rho_{r}, q_{r}, u_{r}),\,\,\,\, \qquad \qquad x>0,
 	\end{cases}
 \end{equation} where $\epsilon>0.$ Similar to the case (b), one can obtain the limiting behavior of the Riemann solution to  \eqref{m1}, \eqref{m2}, and \eqref{m123} as pressure drops to zero. 
 
% Also, $ \rho^{AB, \epsilon}, $ $ \rho^{AB, \epsilon}q^{AB, \epsilon}, $ and $ \rho^{AB, \epsilon}u^{AB, \epsilon} $ converge in the distributional sense and their limit functions can be described as
% \begin{equation*}\label{m124}
% H_{\rho}(x-u_{\ell} t)+\rho_{r}(u_{\ell}-u_{r})t \delta(x-u_{\ell} t),
% \end{equation*}
% \begin{equation*}\label{m129}
% H_{\rho q}(x-u_{\ell} t)+\rho_{r}q_{r}(u_{\ell}-u_{r})t \delta(x-u_{\ell} t),
% \end{equation*}and
% \begin{equation*}\label{m130}
% H_{\rho u}(x-u_{\ell} t)+\rho_{r}u_{r}(u_{\ell}-u_{r})t \delta(x-u_{\ell} t),
% \end{equation*}
% respectively.

 \subsection{Case (ii)  $\boldsymbol{q_{\ell}<q_{r}}$}  From \eqref{m40}, one has
 \begin{equation}\label{m131}
 	\mathtt{f}_{1}(u_{r})=\mathtt{h}_{1}(\mathtt{g}_{1}^{-1}(u_{r}))=\mathcal{P}(\mathtt{g}_{1}^{-1}(u_{r}))+q_{\ell}=\mathcal{P}(\rho_{c})+q_{r}\to \rho_{\ell}(u_{\ell}-u_{r})^2+q_{\ell} \quad \text{as} \quad A, B \to 0,
 \end{equation}
 and \begin{equation}\label{m132}
 	\mathtt{f}_{2}(u_{r})=\mathtt{h}_{2}(\mathtt{g}_{2}^{-1}(u_{r}))=\mathcal{P}(\mathtt{g}_{2}^{-1}(u_{r}))+q_{r} =\mathcal{P}(\rho_{r})+q_{r}\to  q_{r} \quad \text{as} \quad A, B \to 0,
 \end{equation}where $\rho_{c}$ satisfies
 \begin{equation}\label{m133}
 	u_{r}=u_{\ell}-\sqrt{\mathcal{F}(\rho_{\ell}, \rho_{c})}, \qquad \rho_{c}>\rho_{\ell}.
 \end{equation}
 This discussion will again be divided into the following three sub-cases:
 \\\textbf{(a).} $ \boldsymbol{ q_{r} < \rho_{\ell}(u_{\ell}-u_{r})^2 +q_{\ell}}:$
 In this case, there exists $ m_{3}\in (0, m_{o}) $ such that $ \mathtt{f}_{1}(u_{r})> \mathtt{f}_{2}(u_{r})$ for $ A, B \in (0, m_{3}). $ Now, for fixed $ A, B \in (0, m_{3}), $ the Riemann solution is composed by back shock wave, contact discontinuity, and front shock wave. Limiting behavior of the Riemann solution to  
 \eqref{m1}-\eqref{m2} is same as the case (a) for $q_{\ell}>q_{r}.$
  \\\textbf{(b).} $ \boldsymbol{ q_{r}>\rho_{\ell}(u_{\ell}-u_{r})^2 +q_{\ell}}:$ In this case, there exists $ m_{4}\in (0, m_{o}) $ such that $ \mathtt{f}_{1}(u_{r})< \mathtt{f}_{2}(u_{r})$ for $ A, B \in (0, m_{4}). $ Now, for fixed $ A, B \in (0, m_{4}), $ the Riemann solution is composed by back shock wave, contact discontinuity, and front rarefaction wave. Similar to the case (b) for $q_{\ell} > q_{r},$ the intermediate density between the back shock wave and contact discontinuity tends to an extreme concentration and the intermediate state between the contact discontinuity and the front rarefaction wave converges to a vacuum state as pressure drops to zero.
 
\noindent \textbf{(c).} $ \boldsymbol{ q_{r} = \rho_{\ell}(u_{\ell}-u_{r})^2 +q_{\ell}}:$
 In order to deal with this case, we perturbed our initial data \eqref{m3} as follows
 \begin{equation} \label{m134}
 	(\rho, q,  u)_{t=0} =	\begin{cases}
 		(\rho_{\ell}, q_{\ell}, u_{\ell}),\qquad \qquad\qquad x<0,
 		\\ (\rho_{r}, q_{r}^{\epsilon}=q_{r}+\epsilon, u_{r}), \qquad\,  x>0,
 	\end{cases}
 \end{equation} where $\epsilon>0.$ Similar to the case (b), one can get the limiting behavior of the Riemann solution to  \eqref{m1},\eqref{m2}, and \eqref{m134} as pressure drops to zero.

 \begin{remark}
 \textbf{Limiting behaviors of the Riemann solutions for $ u_{\ell}\leq u_{r}.$} We analyze the limiting behaviors of the Riemann solutions for $ u_{\ell}\leq u_{r} $ as pressure vanishes, \textit{i.e.,} $ A, B \to 0.  $ It may be noticed that for $ u_{\ell}< u_{r},$ there exists  $ \bar{m}_{o} $ such that $ (\rho_{AB}, u_{AB}) \in  \mathcal{R}_{q}^{\ell}$ and $ (\rho_{AB}, u_{AB}) \in  \mathcal{R}_{q}^{r}$ for any $ A, B \in (0, \bar{m}_{o}), $ where $ (\rho_{AB}, u_{AB})$ is the point of intersection between $ \mathscr{W}_{q}^{\ell}$ and $ \mathscr{W}_{q}^{r}.$ Thus, we only need to consider the case \bm{$\mathcal{B}$}. Additionally, for $ u_{\ell}= u_{r},$ one only needs to consider the cases \bm{$\mathcal{C}$} and \bm{$\mathcal{D}$}. Depending on the relations between $ q_{\ell} $ and $ q_{r}, $ $ \mathtt{f}_{1}(u_{\ell}) $ and $ \mathtt{f}_{2}(u_{\ell}), $ $ \mathtt{f}_{1}(u_{r}) $ and $ \mathtt{f}_{2}(u_{r}), $ the limiting behaviors of the Riemann solutions to \eqref{m1}-\eqref{m2} can be analyzed in a similar manner as we have established for $ u_{\ell}> u_{r}.$\end{remark}

 \section{Numerical simulations} 
  In section $3, $ we have shown that the Riemann solution composed by two shock waves and a contact discontinuity converges to a single delta shock solution as pressure drops to zero. In order to validate the formation of single delta shock solution, we present numerical computations based on the fifth order weighted essentially non-oscillatory (WENO) scheme. The computational domain is discretized using a uniform spatial step size, $\Delta x=0.005$, while the time step is not prescribed a priori. Instead, it is dynamically adjusted at each time level according to the CFL condition
  \begin{equation}
  	\Delta t=
  	\mathrm{CFL}
  	\frac{\Delta x}{\displaystyle\max_i|\lambda_i|},
  \end{equation}
  where $(\lambda_i)$ denotes the eigenvalues of the system. In all numerical computations, we set the CFL number to $\mathrm{CFL}=0.475.$ For $u_{\ell}>u_{r},$  $q_{\ell}>q_{r},$ and  $q_{\ell}<\rho_{r}(u_{\ell}-u_{r})^2+q_{r},$ let us consider the following initial data

  \begin{figure}[htbp]
  	\centering
  	
  	%------------------------------------------------
  	% Figure 1
  	%------------------------------------------------
  	\begin{minipage}{0.32\textwidth}
  		\centering
  		\begin{tikzpicture}
  			
  			% Main figure
  			\node[anchor=south west] (main1) at (0,0)
  			{\includegraphics[width=\textwidth]{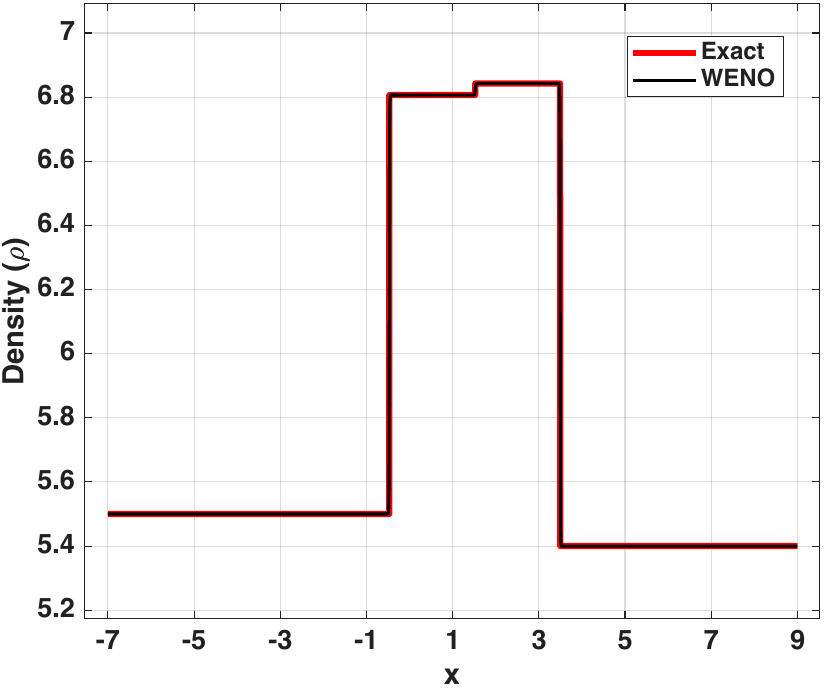}};
  			
  			% Zoomed inset
  			\node[anchor=south west, inner sep=0pt]
  			(inset1) at (3.5,2.23)
  			{\includegraphics[width=0.28\textwidth]{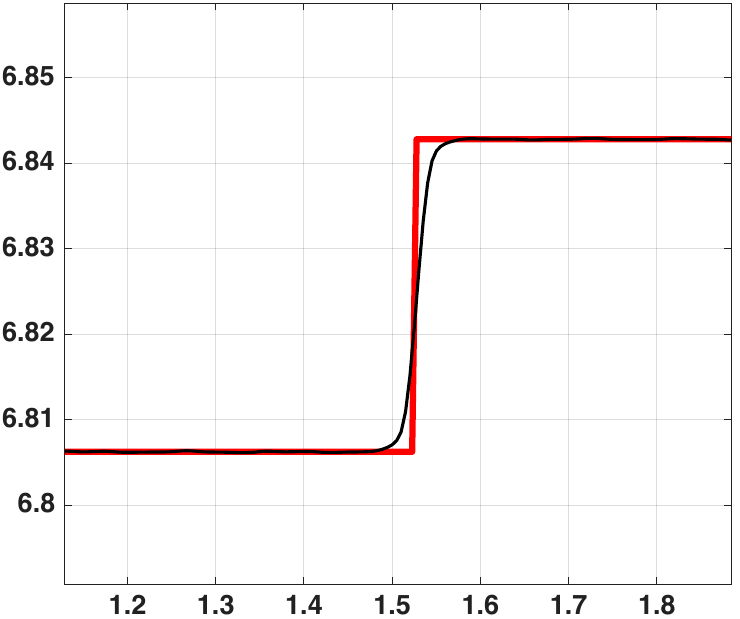}};

  		\end{tikzpicture}
  	\end{minipage}
  	\hfill
  	%------------------------------------------------
  	% Figure 2
  	%------------------------------------------------
  	\begin{minipage}{0.32\textwidth}
  		\centering
  		\begin{tikzpicture}
  			
  			% Main figure
  			\node[anchor=south west] (main2) at (0,0)
  			{\includegraphics[width=\textwidth]{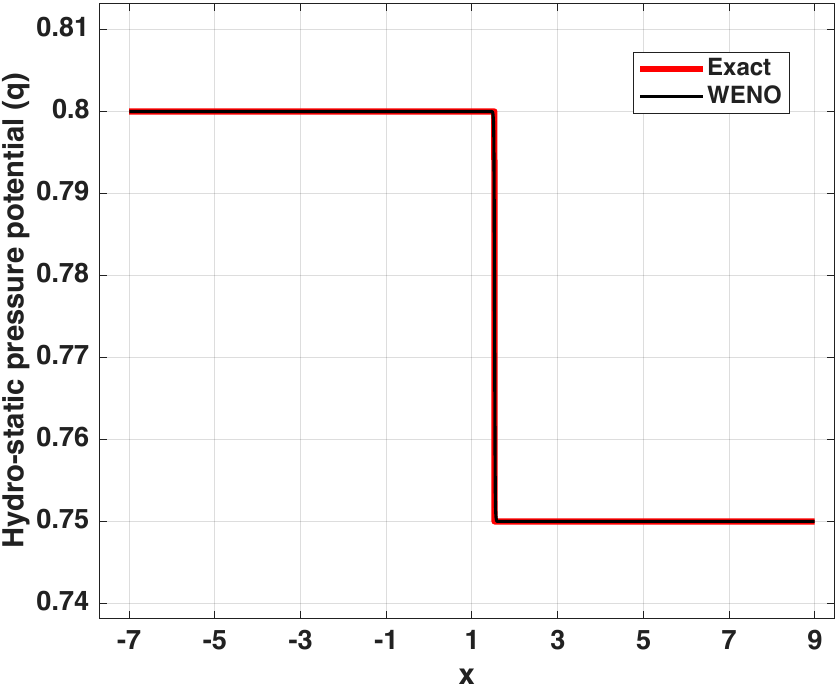}};
  			
  			% Zoomed inset
  			\node[anchor=south west, inner sep=0pt]
  			(inset2) at (3.5,2.13)
  			{\includegraphics[width=0.28\textwidth]{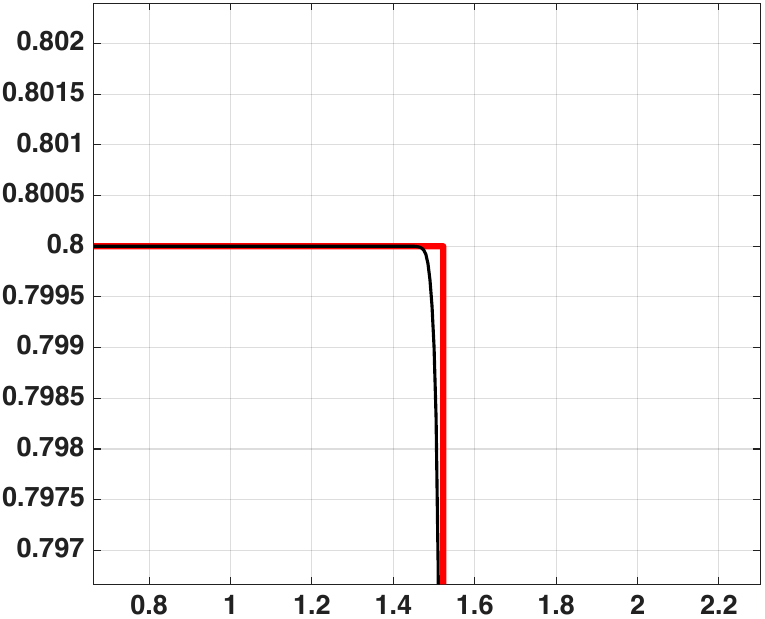}};

  		\end{tikzpicture}
  	\end{minipage}
  	\hfill
  	%------------------------------------------------
  	% Figure 3
  	%------------------------------------------------
  	\begin{minipage}{0.32\textwidth}
  		\centering
  		\begin{tikzpicture}
  			
  			% Main figure
  			\node[anchor=south west] (main3) at (0,0)
  			{\includegraphics[width=\textwidth]{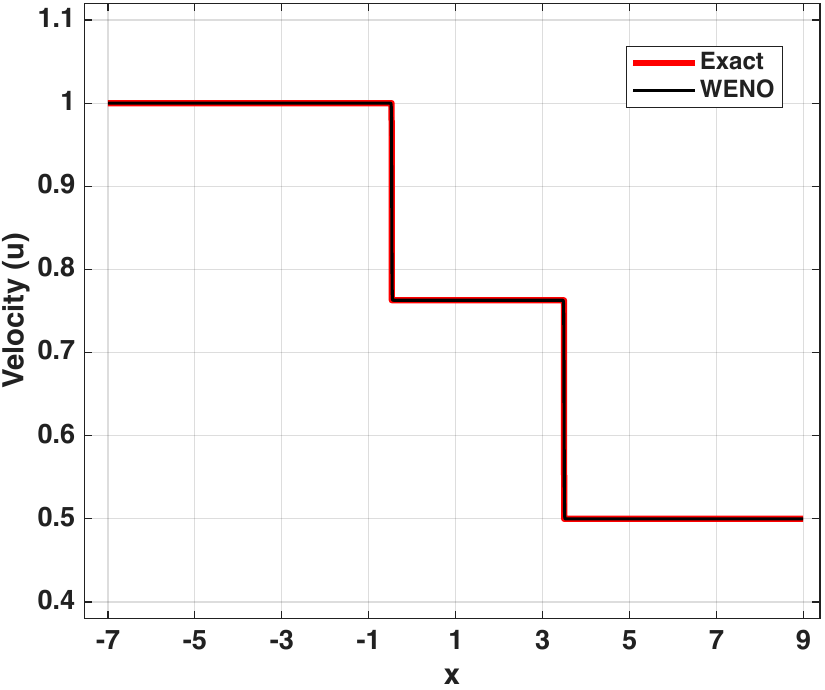}};
  			
  			% Zoomed inset
  			\node[anchor=south west, inner sep=0pt]
  			(inset3) at (3.5,2.3)
  			{\includegraphics[width=0.28\textwidth]{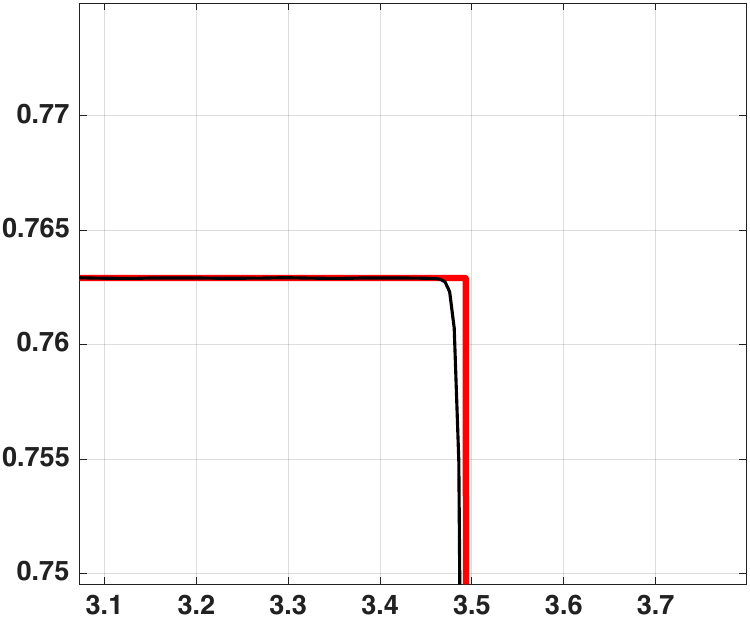}};
  			
  		\end{tikzpicture}
  	\end{minipage}
  	
  	\caption{\mbox{The profiles of state variables $\rho,$ $q,$ and $u$ are displayed for $A=B=0.1.$}}
  	\label{4pf1}
  	
  \end{figure}

\begin{figure}[htbp]
	\centering
	
	%------------------------------------------------
	% Figure 1
	%------------------------------------------------
	\begin{minipage}{0.32\textwidth}
		\centering
		\begin{tikzpicture}
			
			% Main figure
			\node[anchor=south west] (main1) at (0,0)
			{\includegraphics[width=\textwidth]{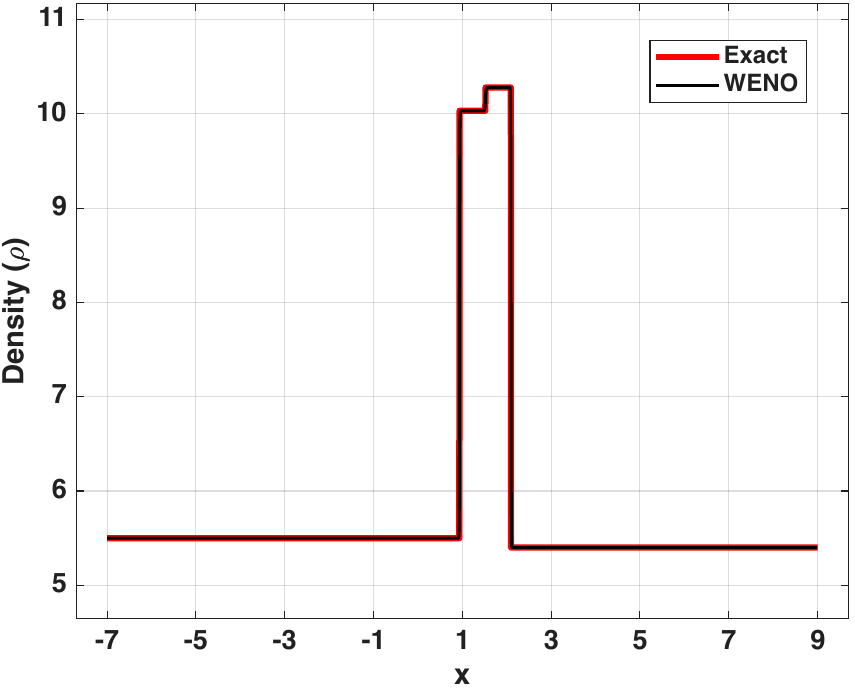}};
			
			% Zoomed inset
			\node[anchor=south west, inner sep=0pt]
			(inset1) at (3.5,2.1)
			{\includegraphics[width=0.28\textwidth]{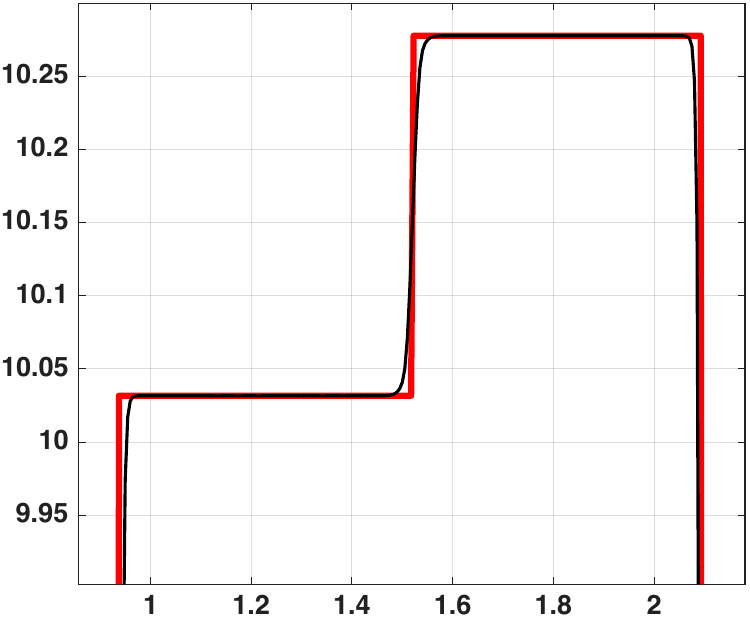}};

		\end{tikzpicture}
	\end{minipage}
	\hfill
	%------------------------------------------------
	% Figure 2
	%------------------------------------------------
	\begin{minipage}{0.32\textwidth}
		\centering
		\begin{tikzpicture}
			
			% Main figure
			\node[anchor=south west] (main2) at (0,0)
			{\includegraphics[width=\textwidth]{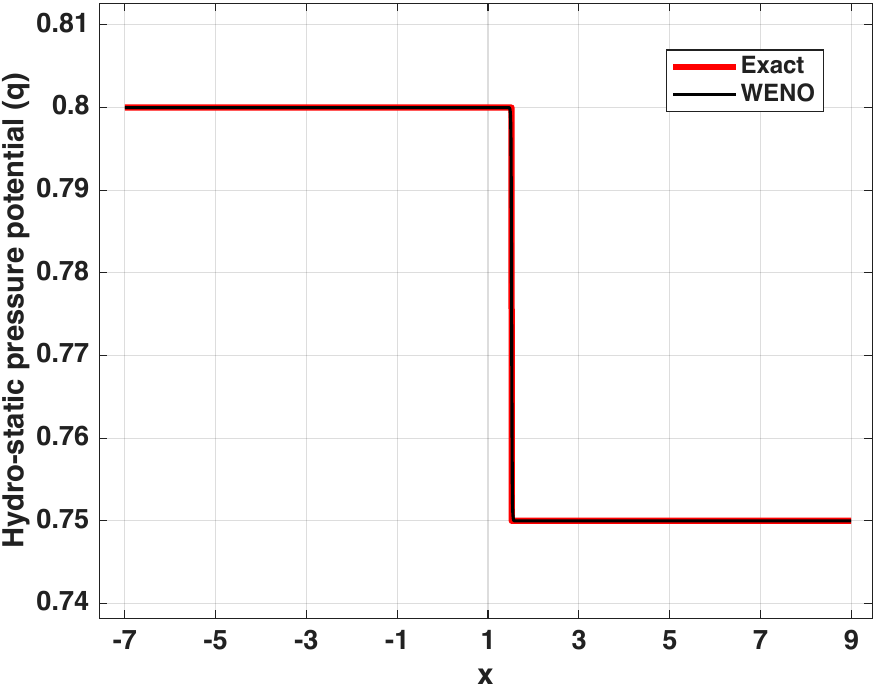}};
			
			% Zoomed inset
			\node[anchor=south west, inner sep=0pt]
			(inset2) at (3.5,2)
			{\includegraphics[width=0.28\textwidth]{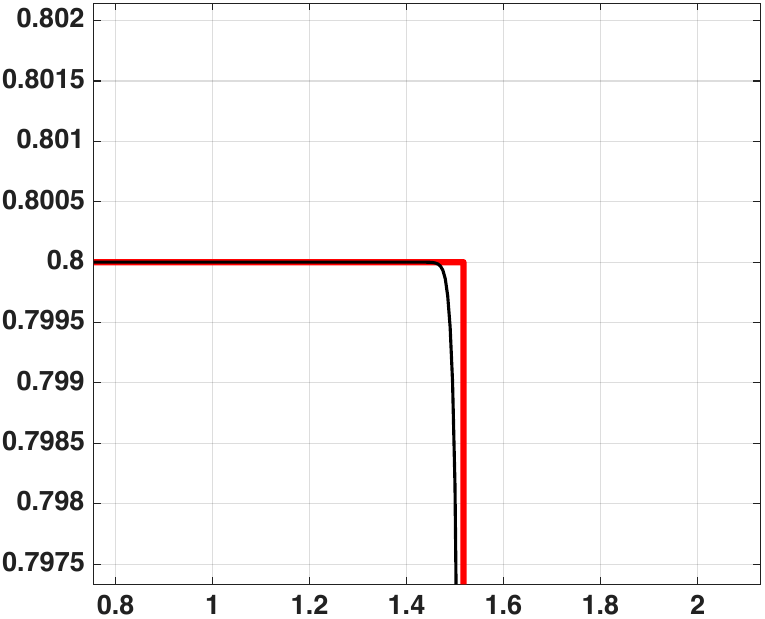}};

		\end{tikzpicture}
	\end{minipage}
	\hfill
	%------------------------------------------------
	% Figure 3
	%------------------------------------------------
	\begin{minipage}{0.32\textwidth}
		\centering
		\begin{tikzpicture}
			
			% Main figure
			\node[anchor=south west] (main3) at (0,0)
			{\includegraphics[width=\textwidth]{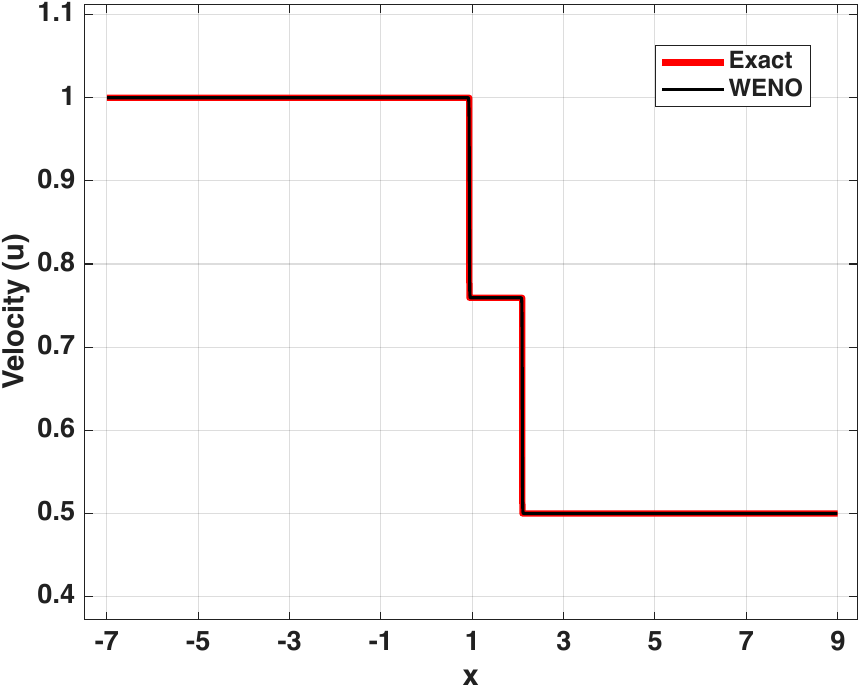}};
			
			% Zoomed inset
			\node[anchor=south west, inner sep=0pt]
			(inset3) at (3.4,2)
			{\includegraphics[width=0.28\textwidth]{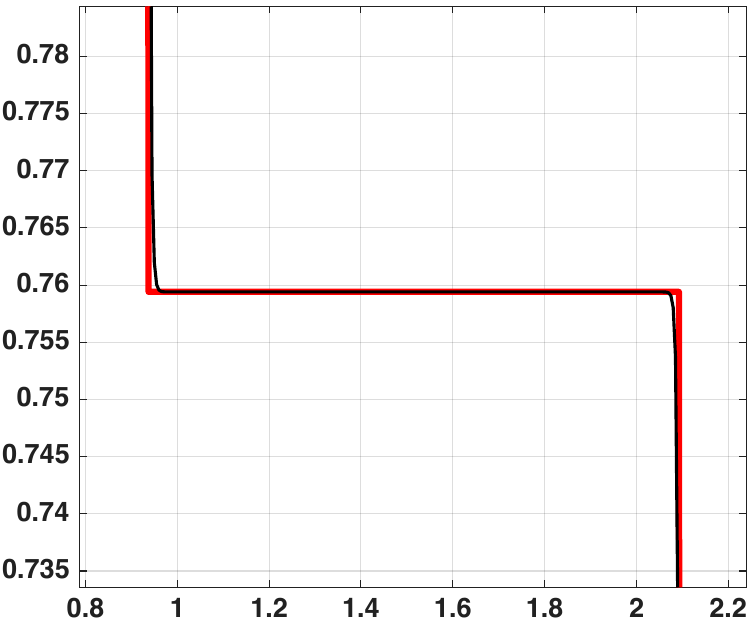}};
			
		\end{tikzpicture}
	\end{minipage}
	
	\caption{\mbox{The profiles of state variables $\rho,$ $q,$ and $u$ are represented for $A=B=0.01.$}}
	\label{4pf2}
	
\end{figure}

\begin{figure}[htbp]
	\centering
	
	%------------------------------------------------
	% Figure 1
	%------------------------------------------------
	\begin{minipage}{0.32\textwidth}
		\centering
		\begin{tikzpicture}
			
			% Main figure
			\node[anchor=south west] (main1) at (0,0)
			{\includegraphics[width=\textwidth]{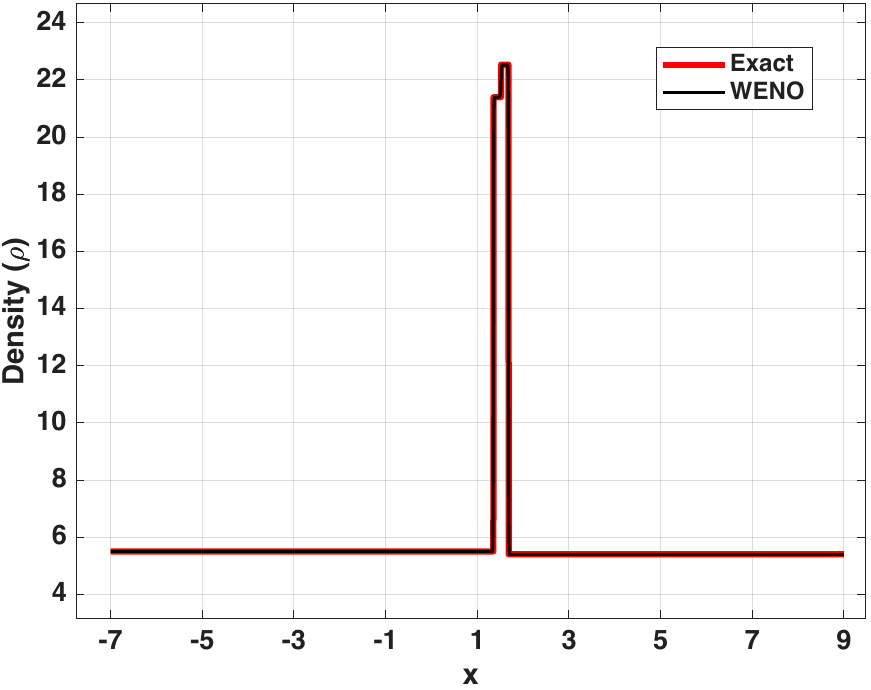}};
			
			% Zoomed inset
			\node[anchor=south west, inner sep=0pt]
			(inset1) at (3.5,2)
			{\includegraphics[width=0.28\textwidth]{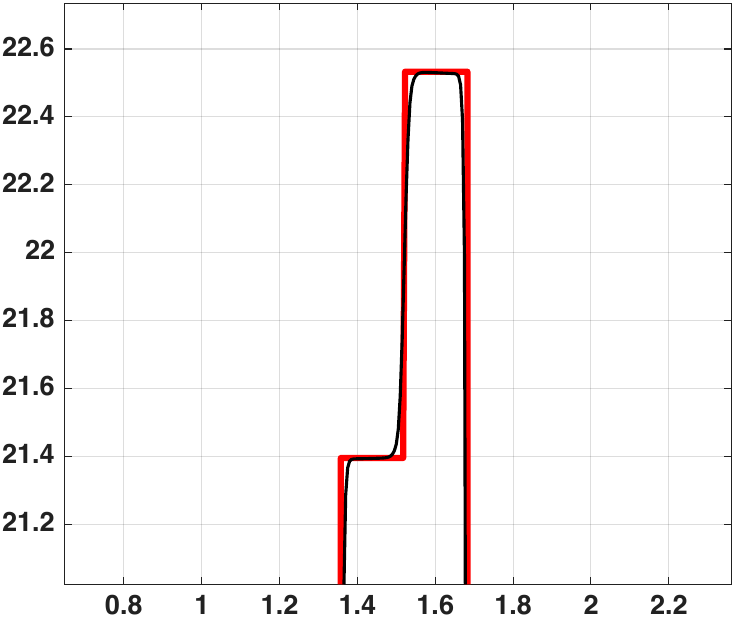}};

		\end{tikzpicture}
	\end{minipage}
	\hfill
	%------------------------------------------------
	% Figure 2
	%------------------------------------------------
	\begin{minipage}{0.32\textwidth}
		\centering
		\begin{tikzpicture}
			
			% Main figure
			\node[anchor=south west] (main2) at (0,0)
			{\includegraphics[width=\textwidth]{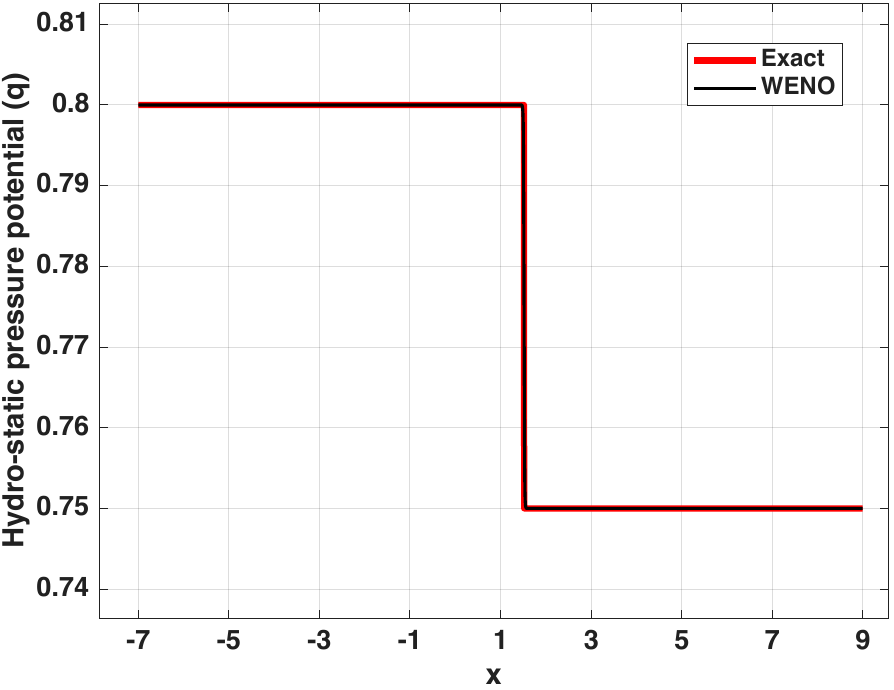}};
			
			% Zoomed inset
			\node[anchor=south west, inner sep=0pt]
			(inset2) at (3.5,2)
			{\includegraphics[width=0.28\textwidth]{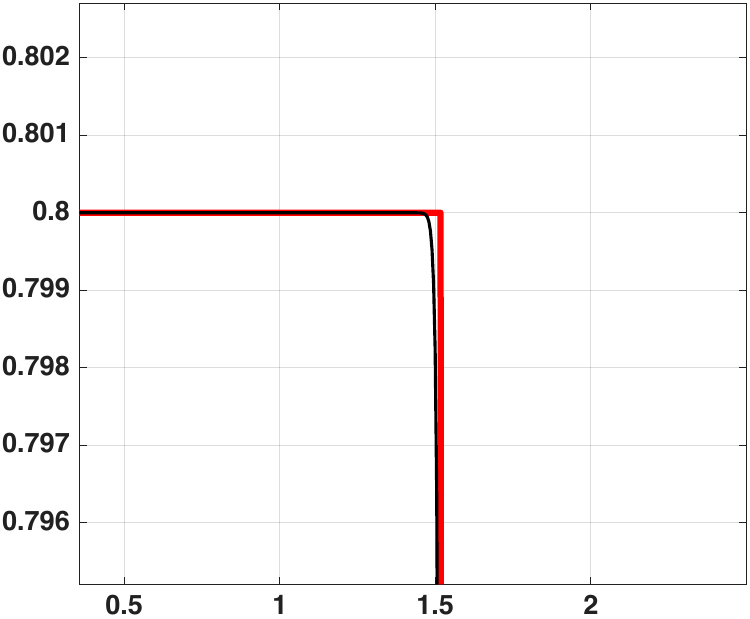}};

		\end{tikzpicture}
	\end{minipage}
	\hfill
	%------------------------------------------------
	% Figure 3
	%------------------------------------------------
	\begin{minipage}{0.32\textwidth}
		\centering
		\begin{tikzpicture}
			
			% Main figure
			\node[anchor=south west] (main3) at (0,0)
			{\includegraphics[width=\textwidth]{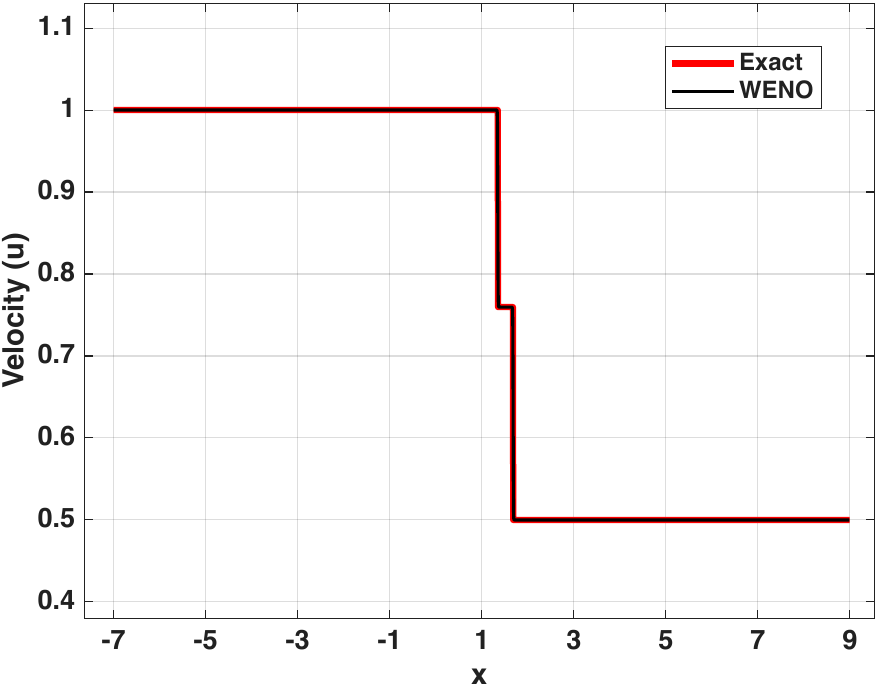}};
			
			% Zoomed inset
			\node[anchor=south west, inner sep=0pt]
			(inset3) at (3.4,2)
			{\includegraphics[width=0.28\textwidth]{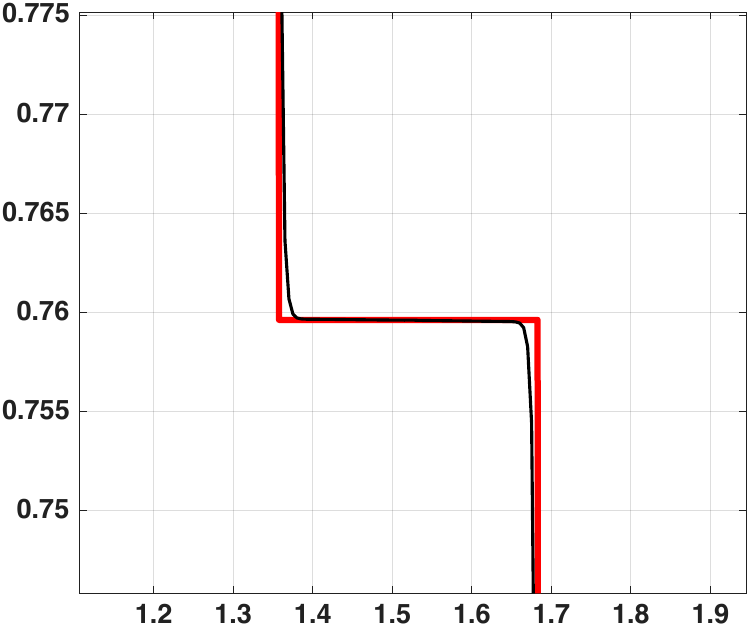}};
		
		\end{tikzpicture}
	\end{minipage}
	
	\caption{\mbox{The profiles of state variables $\rho,$ $q,$ and $u$ are shown for $A=B=0.001.$}}
	\label{4pf3}
	
\end{figure}

\begin{figure}[htbp]
	\centering
	
	%------------------------------------------------
	% Figure 1
	%------------------------------------------------
	\begin{minipage}{0.32\textwidth}
		\centering
		\begin{tikzpicture}
			
			% Main figure
			\node[anchor=south west] (main1) at (0,0)
			{\includegraphics[width=\textwidth]{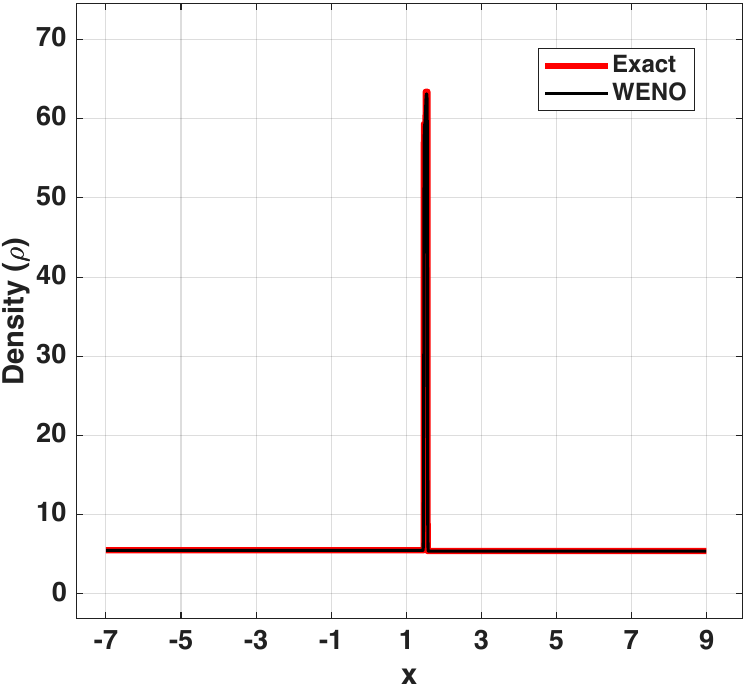}};
			
			% Zoomed inset
			\node[anchor=south west, inner sep=0pt]
			(inset1) at (3.5,2.55)
			{\includegraphics[width=0.28\textwidth]{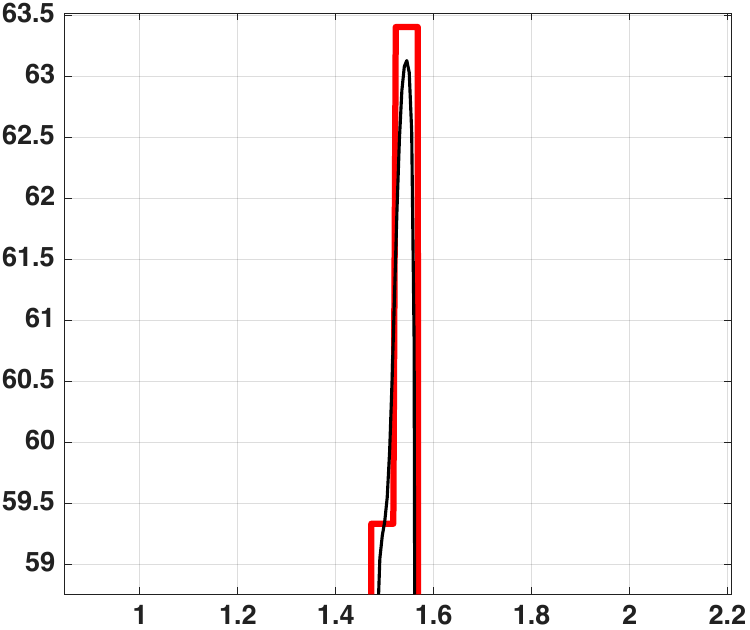}};

		\end{tikzpicture}
	\end{minipage}
	\hfill
	%------------------------------------------------
	% Figure 2
	%------------------------------------------------
	\begin{minipage}{0.32\textwidth}
		\centering
		\begin{tikzpicture}
			
			% Main figure
			\node[anchor=south west] (main2) at (0,0)
			{\includegraphics[width=\textwidth]{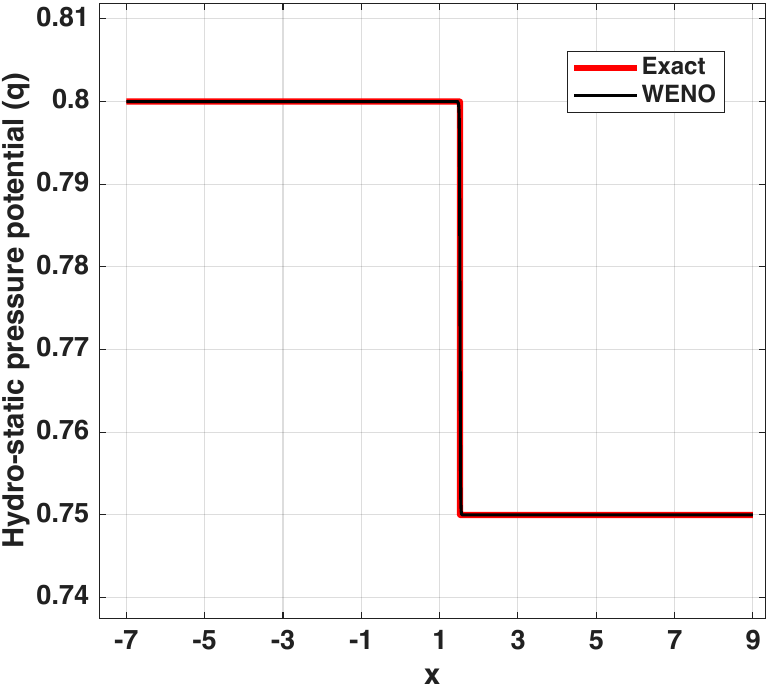}};
			
			% Zoomed inset
			\node[anchor=south west, inner sep=0pt]
			(inset2) at (3.5,2.4)
			{\includegraphics[width=0.28\textwidth]{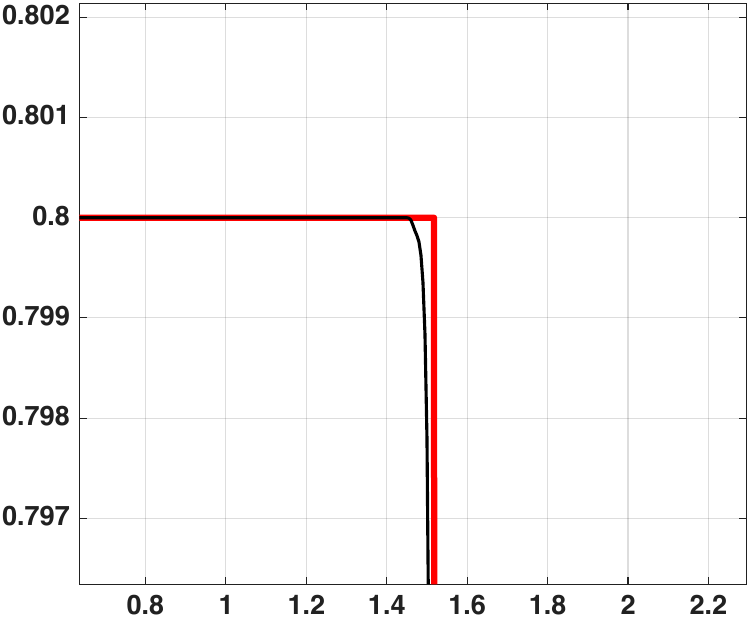}};

		\end{tikzpicture}
	\end{minipage}
	\hfill
	%------------------------------------------------
	% Figure 3
	%------------------------------------------------
	\begin{minipage}{0.32\textwidth}
		\centering
		\begin{tikzpicture}
			
			% Main figure
			\node[anchor=south west] (main3) at (0,0)
			{\includegraphics[width=\textwidth]{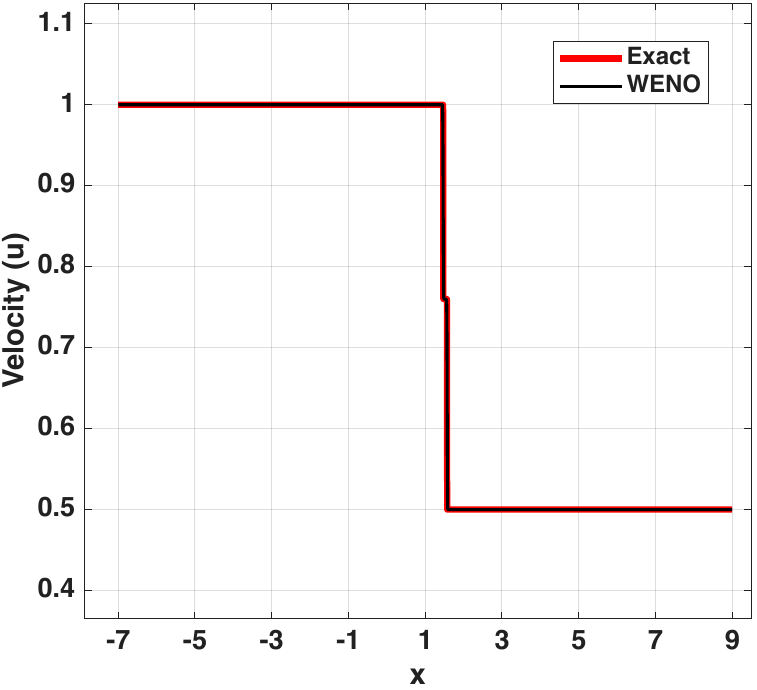}};
			
			% Zoomed inset
			\node[anchor=south west, inner sep=0pt]
			(inset3) at (3.4,2.55)
			{\includegraphics[width=0.28\textwidth]{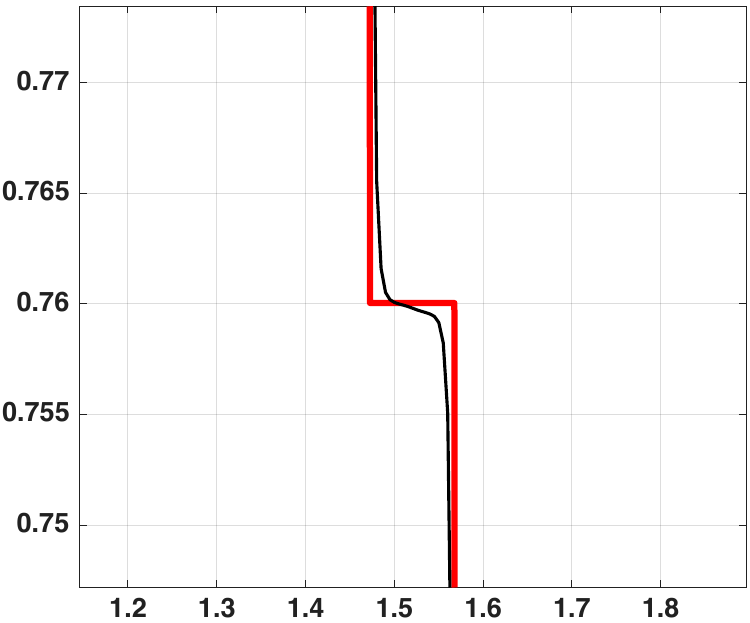}};
			
		\end{tikzpicture}
	\end{minipage}
	
	\caption{\mbox{The profiles of state variables $\rho,$ $q,$ and $u$ are displayed for $A=B=0.0001.$}}
	\label{4pf4}
	
\end{figure}

\begin{equation}\label{h79}
	(\rho_{\ell,r}, q_{\ell,r}, u_{\ell,r})=\begin{cases}
		(5.5, 0.8, 1.0), \,\,\,\qquad x<0,\\
		(5.4, 0.75, 0.5), \qquad x>0,
	\end{cases}
\end{equation}and $\alpha=0.25,$ $n=2.$ Figures~\ref{4pf1}-\ref{4pf5} illustrate the behaviors of the state variables $\rho,$ $q,$ and $u$ for different pairs of values of $A$ and $B.$

%\begin{equation}\label{h79}
%	(\rho_{\ell,r}, q_{\ell,r}, u_{\ell,r})=\begin{cases}
%		(5.5, 0.8, 1.0), \,\,\qquad x<0,\\
%		(5.4, 0.75, 0.5), \qquad x>0,
%	\end{cases}
%\end{equation}and $\alpha=0.25,$ $n=2.$ Figures~\ref{4pf1}-\ref{4pf5} illustrate the behaviors of the state variables $\rho,$ $q,$ and $u$ for different pairs of values of $A$ and $B.$
 %Additionally, Table-$2$ shows the comparison between the numerical and exact values of the intermediate variables corresponding to the different values of $A$ and $B.$ 

\begin{figure}[htbp]
	\centering
	
	%------------------------------------------------
	% Figure 1
	%------------------------------------------------
	\begin{minipage}{0.32\textwidth}
		\centering
		\begin{tikzpicture}
			
			% Main figure
			\node[anchor=south west] (main1) at (0,0)
			{\includegraphics[width=\textwidth]{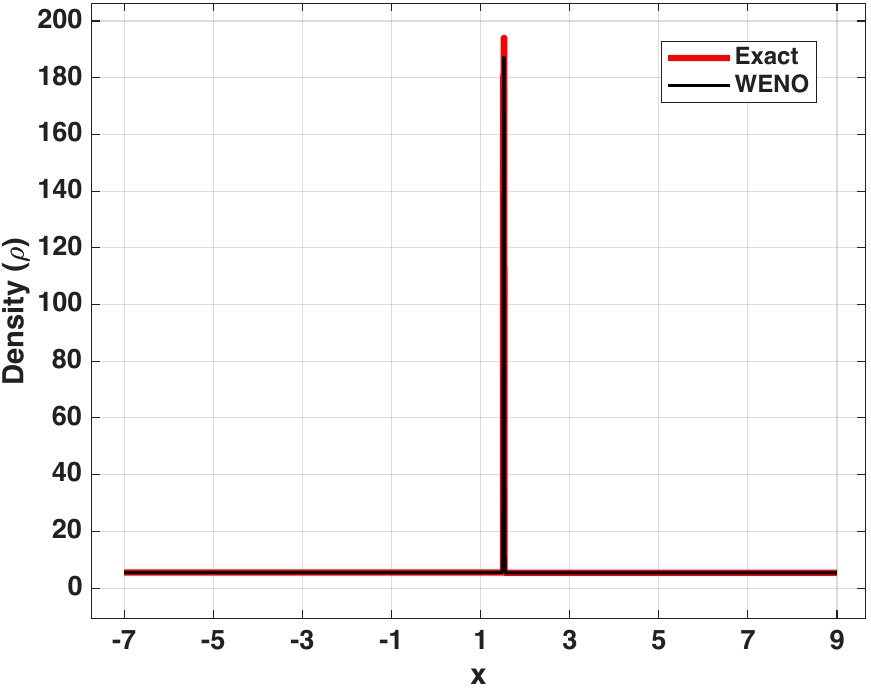}};
			
			% Zoomed inset
			\node[anchor=south west, inner sep=0pt]
			(inset1) at (3.5,2.1)
			{\includegraphics[width=0.28\textwidth]{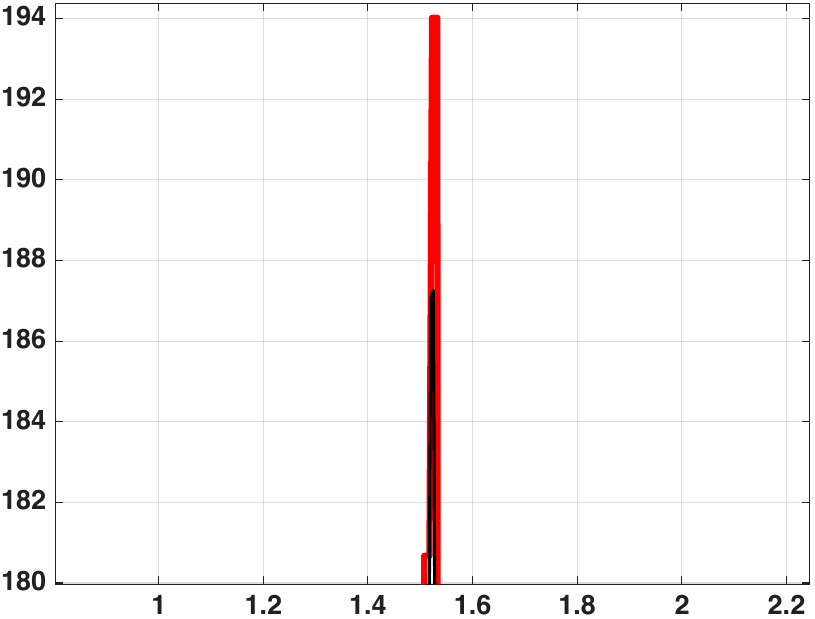}};

		\end{tikzpicture}
	\end{minipage}
	\hfill
	%------------------------------------------------
	% Figure 2
	%------------------------------------------------
	\begin{minipage}{0.32\textwidth}
		\centering
		\begin{tikzpicture}
			
			% Main figure
			\node[anchor=south west] (main2) at (0,0)
			{\includegraphics[width=\textwidth]{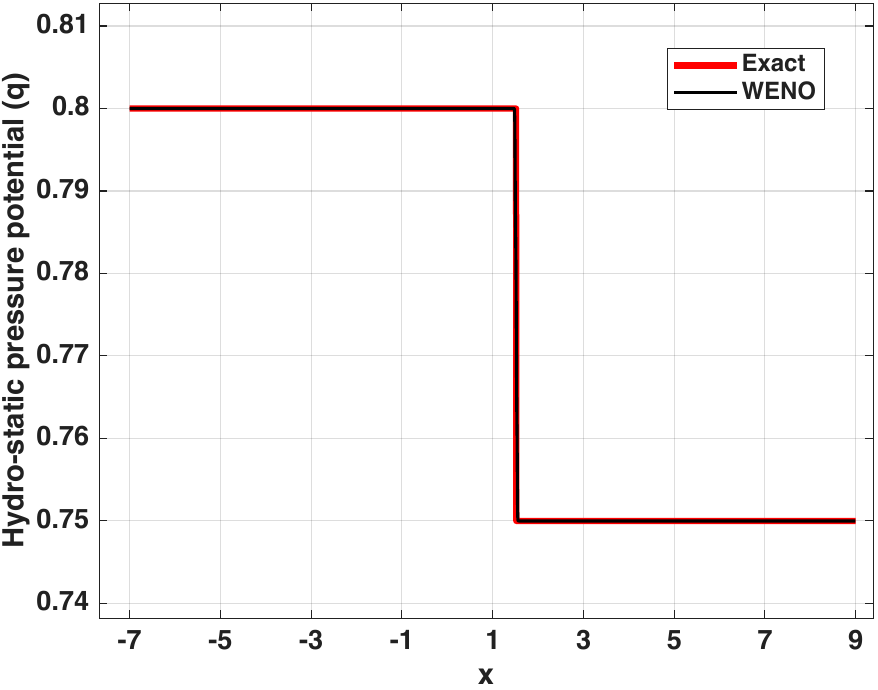}};
			
			% Zoomed inset
			\node[anchor=south west, inner sep=0pt]
			(inset2) at (3.5,2)
			{\includegraphics[width=0.28\textwidth]{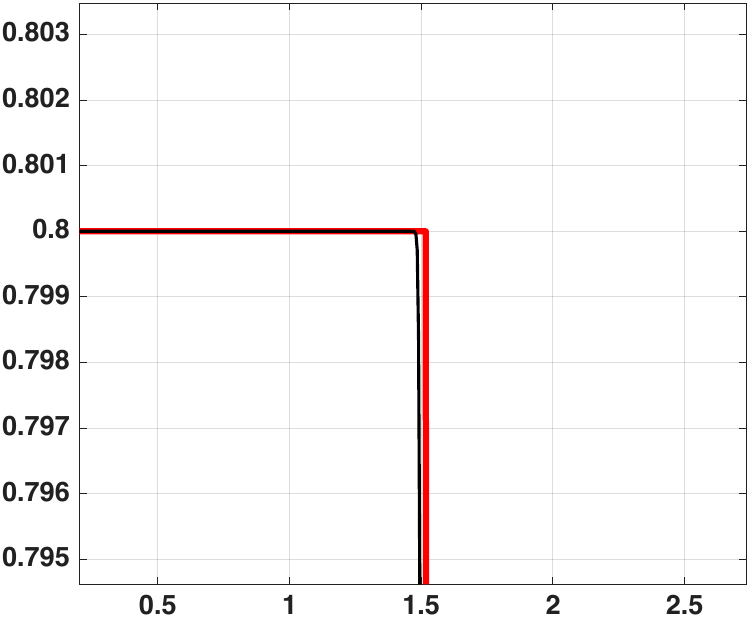}};

		\end{tikzpicture}
	\end{minipage}
	\hfill
	%------------------------------------------------
	% Figure 3
	%------------------------------------------------
	\begin{minipage}{0.32\textwidth}
		\centering
		\begin{tikzpicture}
			
			% Main figure
			\node[anchor=south west] (main3) at (0,0)
			{\includegraphics[width=\textwidth]{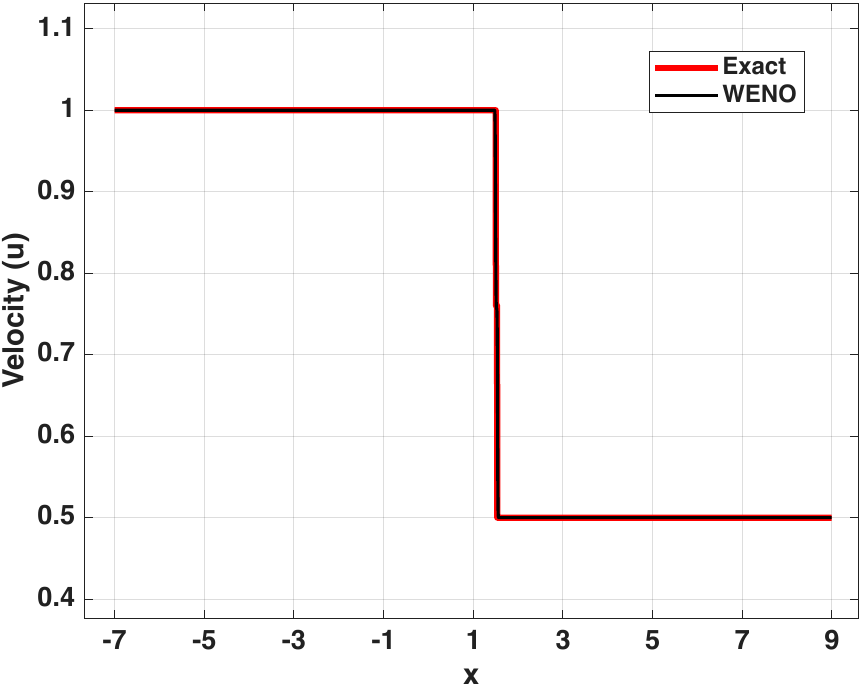}};
			
			% Zoomed inset
			\node[anchor=south west, inner sep=0pt]
			(inset3) at (3.4,2)
			{\includegraphics[width=0.28\textwidth]{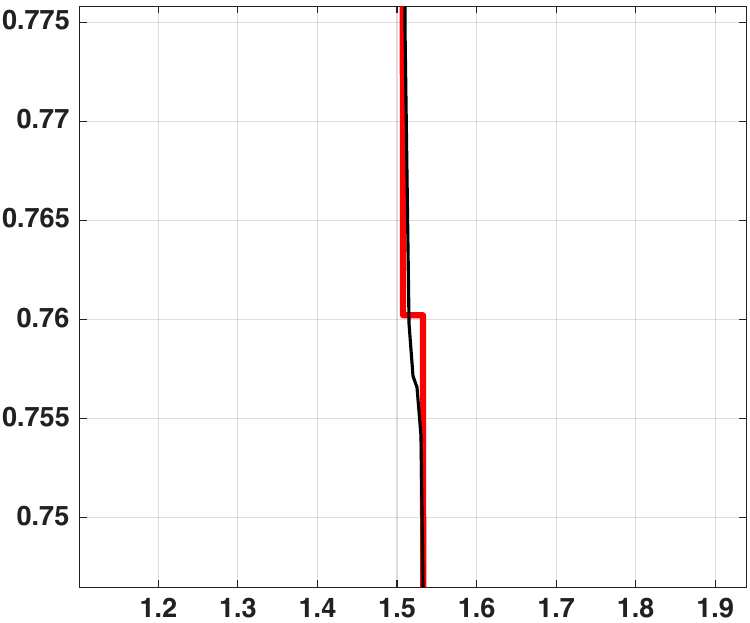}};
			
		\end{tikzpicture}
	\end{minipage}
	
	\caption{\mbox{The profiles of state variables $\rho,$ $q,$ and $u$ are shown for $A=B=0.00001.$}}
	\label{4pf5}
	
\end{figure}

 \section{Conclusion}
 This study examines Riemann solutions of the isentropic Cargo–LeRoux system of conservation laws along with extended Chaplygin fluid models. In fact, it has been shown that Riemann solutions consist four different structures under some special conditions on the initial data. We explore each of the Riemann solutions' limiting behaviors as pressure gets closer to zero. It is shown that in the vanishing pressure limiting case, the Riemann solution having two shock waves and a contact discontinuity tends to a single delta shock solution which is same as the corresponding Riemann solution of the pressureless Cargo-LeRoux model. Further, it is analyzed that the Riemann solution composed by back rarefaction (shock) wave, contact discontinuity, and front shock (rarefaction) wave  converges to the solution having a vacuum state on the left (right) and a delta shock wave on the right (left) for the vanishing pressure limit, respectively. It may be remarked that if the hydro-static pressure is  constant then the Riemann solutions of isentropic Cargo-LeRoux reduce to the corresponding Riemann solutions of Euler equations with same extended Chaplygin fluids. 
 \begin{ack*}The support provided by UGC, India (Ref. No. 201610063559)  is gratefully acknowledged by Priyanka.
 	% The authors are  grateful  to the anonymous referee for the valuable comments/suggestions, which have greatly improved the 
 %	earlier version of this paper. 
 \end{ack*} 
 \noindent\textbf{\textit{Data availability}.} The data that support the findings of this study are available within the article.

% Thus, the delta shock wave with the Dirac delta function in the state variable $\rho$ will appear in the solution. Hence, in the vanishing-parameters limiting case, the Riemann solution composed by $ 1 $-shock wave, contact discontinuity, and $ 3 $-shock wave converges to a single delta shock solution which is same as the
 %corresponding Riemann solution of the pressureless Cargo-LeRoux model.
\def\cprime{$'$}

\end{document}